\documentclass[11pt,reqno]{amsart}

\usepackage[utf8]{inputenc}
\usepackage[margin=1.25in]{geometry}
\usepackage{hyperref}
\usepackage{appendix}
\usepackage{amsfonts}
\usepackage{amsthm}
\usepackage{amssymb}
\usepackage{stmaryrd} 
\usepackage{enumitem}
\usepackage{amsmath}
\usepackage{amsthm}
\usepackage[dvipsnames]{xcolor}
\usepackage{mathrsfs}
\usepackage{lipsum} % for filler text

\theoremstyle{plain}
\newtheorem{theorem}{Theorem}[section]
\newtheorem{corollary}[theorem]{Corollary}
\newtheorem{lemma}[theorem]{Lemma}
\newtheorem{Proposition}[theorem]{Proposition}

\newtheorem{Definition}[theorem]{Definition}

\newtheorem{fact}[theorem]{Fact}

\theoremstyle{remark}
\newtheorem{example}[theorem]{Example}

\newtheorem{Question}[theorem]{Question}

\newtheorem{remark}[theorem]{Remark}

\usepackage{biblatex} %Imports biblatex package
\numberwithin{equation}{section}
\title[Brown measure for polynomials in Ginibre matrices]{Brown measure convergence for the spectrum of polynomials in Ginibre matrices}

\author{Yi HAN}
\address{Institute for Advanced Study, 1 Einstein Drive, Princeton, NJ
}
\email{hanyi@ias.edu}

\begin{document}

\begin{abstract}

Fix a multivariate polynomial $\mathfrak{p}$ in $n$ non-commuting variables of arbitrary degree, and consider $n$ independent $N\times N$ complex Ginibre matrices $X_1^N,\cdots,X_n^N$.  We prove that the empirical spectral distribution of $P^N=\mathfrak{p}(X_1^N,\cdots,X_n^N)$ converges as $N$ tends to infinity to the so-called Brown measure of $\mathfrak{p}$ evaluated at free circular variables. For polynomials of degree at most 2, the convergence was proven by Cook, Guionnet, and Husson \cite{cook2022spectrum}, and we prove that the convergence in fact holds for polynomials $\mathfrak{p}$ of any degree. The main step in the proof is a least singular value lower bound for $P^N-z$ for almost all complex shifts $z$, and we prove this via a least singular value lower bound for a wide class of tensorized Ginibre matrices of finite type with a deterministic shift, which is of independent interest. We further show that the Brown measure convergence holds beyond Gaussians: the same convergence holds when the entry law has mean 0, variance 1, bounded density on $\mathbb{C}$ and finite moments of all orders.

\end{abstract}

\maketitle

\section{Introduction}

For an $N\times N$ matrix $A$ with complex entries and complex eigenvalues $\lambda_1(A),\cdots,\lambda_N(A)$, its empirical spectral distribution (ESD) is defined as the following probability measure 
\begin{equation}
    \mu_A:=\frac{1}{N}\sum_{j=1}^N\delta_{\lambda_j(A)}.
\end{equation}
A central problem in random matrix theory is to study, for a sequence of $N\times N$ matrices $(A^N)_{N\geq 1}$, the convergence of its ESD $\mu_{A^N}$ to some deterministic probability measure $\mu$. When $A_N$ is a real symmetric or Hermitian matrix, its eigenvalues are real and the convergence of $\mu_{A^N}$ can be addressed via the method of moments. For matrices with i.i.d. entries above the diagonal, \cite{wigner1958distribution} established the celebrated Wigner semicircle law; and for the Gram matrix of i.i.d. rectangular matrices, \cite{marvcenko1967distribution} established the celebrated Marchenko–Pastur law.

The convergence of ESDs for non-Hermitian matrix ensembles is significantly harder to establish. For the i.i.d. ensemble $X^N$ where the entries are i.i.d. without imposing symmetry, the convergence of $\mu_{\frac{1}{\sqrt{N}}X^N}$ to the circular law $\mu_{\text{circ}}$ was established in \cite{tao2008random}, \cite{ WOS:000281425000010} and \cite{MR2663633}. A key component among these papers is that we need to derive quantitative least singular value estimates for $\frac{1}{\sqrt{N}}X^N-zI$ over $z\in\mathbb{C}$ in order to rigorously apply Girko's Hermitization trick \cite{girko1985circular} (see Section \ref{section5.1}), which is a new requirement not demanded by Hermitian random matrices. We refer to the survey \cite{MR2908617} for more discussion.

Since then, the circular law has also been proven for sparse random matrices \cite{MR3945840} and sparse random directed graphs \cite{MR3945840}, and the limiting ESD for sparse random graphs with constant average degree was established in \cite{sah2023limiting}. The circular law was also recently proven for many inhomogeneous non-Hermitian random matrices, see \cite{cook2018lower}, \cite{jain2021circular}, \cite{tikhomirov2023pseudospectrum}, \cite{han2024circular}, \cite{han2025circular}, \cite{han2025circular2}. However, in these works, the assumption that the entries are (mostly) independent plays a crucial role for establishing least singular value lower bounds for $\frac{1}{\sqrt{N}}A_N-z$, and thus for establishing the circular law limit.

In this paper, we are specifically interested in random matrices that arise from evaluating a non-commutative polynomial at independent Ginibre matrices. The entries of these matrices are highly dependent.  Then for any polynomial $\mathfrak{p}$ with complex coefficients and $n$ non-commuting variables, and $n$ i.i.d. random matrices $X_1^N,\cdots,X_n^N$, we study the spectrum of $P^N=\mathfrak{p}(X_1^N,\cdots,X_n^N)$. Prior to this work, the convergence of the ESD of $P^N$ was only established in several special classes of polynomial $\mathfrak{p}$.
The very special case of product of i.i.d. matrices was studied in much depth in \cite{tikhomirov2011asymptotics}, \cite{kosters2015limiting}, \cite{o2011products}, and see also \cite{o2015products}. For more general polynomials $\mathfrak{p}$, the convergence of ESD of $P^N$ was only established in \cite{cook2022spectrum} when the polynomial $\mathfrak{p}$ has maximal degree 2. Before the present work, no general theory was known when $\deg\mathfrak{p}\geq 3$. A series of works \cite{sniady2002random}, \cite{guionnet2014convergence}, \cite{wood2016universality} have shown that for a sequence of non-normal matrices converging in $*$-moments (this covers the case of $P^N$ here, see the next paragraph), it is often sufficient to add a vanishing noisy perturbation to regularize the ESD and guarantee convergence of ESD to the perceived limiting measure. However, as a perturbation is needed, this does not constitute a general theory of ESDs for $P^N$ themselves.

Polynomials in random matrices are also central topics in the theory of free probability and operator algebra, stemming from Voiculescu's work \cite{voiculescu1991limit}. Voiculescu showed that free probability theory describes the limiting measures for ESDs of polynomials in Wigner matrices. Indeed, let $\mathfrak{p}$ be a non-commuting and self-adjoint polynomial, and let $W_1^N,\cdots,W_n^N$ be independent $N\times N$ GUE matrices, then the ESD of $\mathfrak{p}(W_1^N,\cdots,W_n^N)$ converges to the spectral distribution of $\mathfrak{p}(s_1,\cdots,s_n)$ where $s_1,\cdots,s_n$ are freely independent semi-circular elements in a $C^*$-algebra. A later notion of strong convergence introduced by Haagerup and Thorbjørnsen \cite{haagerup2005new} further characterized the convergence of the operator norm of the matrix polynomial to the operator norm of the free probability limit. Some more recent progress on this topic can be found in \cite{bandeira2023matrix}, \cite{brailovskaya2024universality}, and see also the survey \cite{van2025strong}. However, all these works concern self-adjoint polynomials in Wigner matrices, and the non-normal case is completely untouched. For instance, \cite{van2025strong}, Section 6.9 asked whether an analogous theory of asymptotic freeness (namely, convergence of the ESD to the Brown measure limit, see Section \ref{section5.1}) can be derived for polynomials $\mathfrak{p}$ evaluated in non-Hermitian random matrices. The main contribution of this paper is to make a significant step in this problem, proving convergence of ESDs for $\mathfrak{p}$ evaluated in Ginibre matrices with arbitrary $\deg\mathfrak{p}$, and the convergence goes beyond Gaussians covering a class of generic distributions with bounded density. The main results of this paper are stated as follows.

We write $\mathbb{C}\langle x_1,\cdots,x_n\rangle$ for the class of polynomials with complex coefficients in $n$ non-commuting indeterminate variables $x_1,\cdots,x_n$.
\begin{Definition}
   A random matrix $X=X^N$ is an $N\times N$ complex Ginibre matrix if the families $(\sqrt{N}\Re(X_{ij}))_{i,j\in[N]}$ and $(\sqrt{N}\Im(X_{ij}))_{i,j\in[N]}$ are independent and i.i.d. families of random variables with distribution $\mathcal{N}(0,\frac{1}{2})$. Here $[N]:=\{1,2,\cdots,N\}$.
\end{Definition}

\begin{theorem}\label{mainconvergencetheorem}
  Let $n\in\mathbb{N}$ and let $\mathfrak{p}\in\mathbb{C}\langle x_1,\cdots,x_n\rangle$ be a non-commutative polynomial of arbitrary degree. Then for each $N\in\mathbb{N}$ let $X_1^N,\cdots,X_n^N$ be independent $N\times N$ complex Ginibre matrices and denote by $P^N=\mathfrak{p}(X_1^N,\cdots,X_n^N)$. Then we have the convergence 
  $$
\mu_{P^N}\to \nu_{\mathfrak{p}(c_1,\cdots,c_n)}
$$ weakly in probability. Here $c_1,\cdots,c_n$ are *-free circular elements in a $W^*$-probability space and $\nu_{\mathfrak{p}(c_1,\cdots,c_n)}$ is the Brown measure of $\mathfrak{p}(c_1,\cdots,c_n)$ (definition spelled out in Section \ref{section5.1}).
\end{theorem}

The main step in proving Theorem \ref{mainconvergencetheorem} is to lower bound the least singular value of shifted versions of $P^N$, which is stated in the following second main result of the paper:

\begin{theorem}\label{theoremonline89}
    Let $n\in\mathbb{N}$ and let $\mathfrak{p}\in\mathbb{C}\langle x_1,\cdots,x_n\rangle$ be a non-commutative polynomial of any degree. Then for any $z\neq\mathfrak{p}(0)$ there are constants $c_0(\mathfrak{p}),c(\mathfrak{p}),c_1(\mathfrak{p})$ depending only on $\mathfrak{p}$, and a constant $\widehat{C}(\mathfrak{p},z)>0$ depending on both $\mathfrak{p}$ and $z$, so that we have the following estimate. Let $X_1^N,\cdots,X_n^N$ be independent $N\times N$ complex Ginibre matrices and denote by $P^N=\mathfrak{p}(X_1^N,\cdots,X_n^N)$. Then for this $z\neq \mathfrak{p}(0)$, we can find some fixed $N_0=N_0(\mathfrak{p},z)$ such that whenever $N\geq N_0$, we have for any $\epsilon>0,$
    $$
\mathbb{P}\{\sigma_{min}(P^N-z)\leq\epsilon
\}\leq \widehat{C}(N^{c_0}\epsilon^c+\exp(-c_1N)).
    $$
\end{theorem}
Here $\sigma_{min}(\cdot)$ is the least singular value of a square matrix and $\mathfrak{p}(0)$ is the constant term in $\mathfrak{p}$, which is obtained from setting all Ginibre matrices $X_i^N$ equal to zero. This least singular value estimate excludes the single point $z=\mathfrak{p}(0)$ and is not uniform in $z$, but is still sufficient for proving the Brown measure convergence in Theorem \ref{mainconvergencetheorem}. We remark that when $\mathfrak{p}$ has degree two, \cite{cook2022spectrum}, Theorem 1.3 proved a least singular value estimate uniformly for all $z\in\mathbb{C}$, but higher degree polynomials are not covered. It is possible to track the exact values of $c_0$ and $c$, and they should depend only on the dimension of the linearization matrix associated with $\mathfrak{p}$, see Section \ref{section2}, but the exact dimension dependence is very complicated so we do not determine the exact value here.

Interestingly, the main results of this paper do not rely on the strict Gaussian structure of Ginibre matrices, and we generalize them to a wide class of non-Gaussian distributions:

\begin{theorem}\label{nongaussiantheorem} Let $n\in\mathbb{N}$ and $\mathfrak{p}\in\mathbb{C}\langle x_1,\cdots,x_n\rangle$.
    Let $\xi$ be a complex-valued random variable with 
    $$
\mathbb{E}[\xi]=0,\quad\mathbb{E}[|\xi|^2]=1,
    $$and assume that $\xi$ has a density $f$ on $\mathbb{C}$ satisfying, for some fixed $K>0$,
    \begin{equation}\label{densityassumption}
\|f\|_{L^\infty(\mathbb{C})}\leq K<\infty.
    \end{equation}Further assume that $\xi$ has finite moments of all orders: for any $Q\in\mathbb{N}_+$ there exists $M_Q<\infty$ such that
    \begin{equation}\label{finitenessofmoments}
\mathbb{E}[|\xi|^Q]\leq M_Q<\infty.
    \end{equation}
    Consider a random matrix $X^N=(\frac{1}{\sqrt{N}}x_{ij})_{1\leq i,j\leq N}$ where the $x_{ij}$ are i.i.d. random variables with law $\xi$. Let $X_1^N,\cdots,X_n^N$ be i.i.d. copies of $X^N$. We similarly denote by $P^N=\mathfrak{p}(X_1^N,\cdots,X_n^N)$. Then the same Brown measure convergence in Theorem \ref{mainconvergencetheorem} also holds for this matrix $P^N$ with non-Gaussian law. 
\end{theorem}
The moment assumption \eqref{finitenessofmoments} on $\xi$ can be weakened to requiring only finiteness up to $Q_\mathfrak{p}$-th moment, for some constant 
$Q_\mathfrak{p}>0$ depending only on the polynomial $\mathfrak{p}$. We still require in \eqref{densityassumption} that $\xi$ has a bounded density on $\mathbb{C}$, and this covers many non-Gaussian distributions such as the uniform distribution over several disjoint disks in the complex plane, or the convolution of a Gaussian distribution with another subgaussian distribution.

\subsection{Main ideas of the proof}
The matrix $P^N$ is not Hermitian, and to prove the convergence of its empirical spectral law, the standard method by now is to use Girko's idea \cite{girko1985circular} to do integration by parts by Green's formula and transfer the problem to the convergence of the integral of $\log(x)$ with respect to the empirical singular value distribution of $P^N-zI$. See Section \ref{section5.2for} for details. The main technical step in justifying the latter convergence is to prove that, for almost every $z\in\mathbb{C}$ with respect to the Lebesgue measure, the least singular value of $P^N-z$ is at least $N^{-C_z}$ for some constant $C_z$. Indeed, aside from this least singular value estimate, all the other steps needed in the proof of Brown measure convergence are relatively standard by now, by combining available technical toolboxes.

For a non-Hermitian square matrix, most existing methods for establishing a least singular value lower bound require independent entries. This is clearly not the case for $P^N$ as it involves a complicated polynomial $\mathfrak{p}$. A well-known method to handle this is to find a linearization matrix $\mathcal{T}_N$, which is a matrix of a larger dimension, such that its least singular value is no larger than the least singular value of $P^N-z$, and such that the least singular value of $\mathcal{T}_N$ is much easier to analyze.
There are several different ways to find a linearization matrix for $P^N-z$ and there is no clear consensus which linearization is the \textit{best} when $\deg\mathfrak{p}\geq 3$. In Section \ref{section2} we introduce one such linearization, and the resulting linearization matrix is written as
$$
\mathcal{L}_z(X)=K_z\otimes I_N+\sum_{\ell=1}^n H_\ell\otimes X_\ell,
$$where $K_z,H_\ell$ are square matrices of dimension $m$ and the dimension $m$ depends only on $\mathfrak{p}$.  

To prove a least singular value lower bound for $\mathcal{L}_z$, a first attempt would be to construct a net for vectors in $\mathbb{C}^{N\times m}$ and show that the Gaussian randomness beats the cardinality of the net. But a major problem arises here: do we know that the coefficients $H_\ell$ span the whole matrix space, or otherwise the covariance matrix of the Gaussian entries lives at a much lower dimension so that we do not have enough randomness at all to beat the net size? Unfortunately, we do not have sufficient structural information about the linearization matrices $H_\ell$ particularly when $\deg\mathfrak{p}\geq 3$, and in most interesting cases, the matrices $H_\ell$ do not span the whole space $M_m(\mathbb{C})$ and a naive net argument via Gaussian anti-concentration is insufficient. 

To simplify the problem, we note that $K_z$ is invertible when $z\neq \mathfrak{p}(0)$, so we multiply by $K_z^{-1}$ to turn $H_\ell$ into $H_\ell K_z^{-1}$ and turn $K_z$ to $I_m$. This transform is legitimate since $K_z$ is singular only at $z=\mathfrak{p}(0)$, and we simply ignore this value of $z$ by proving the estimate at all other $z\in\mathbb{C}$. The benefit is immediate: suppose that on some linear subspace of $\mathbb{C}^m$, each $H_\ell K_z^{-1}$ acts by 0, then the $I_m$ term preserves invertibility; whereas if $H_\ell K_z^{-1}$ acts non-trivially, we may use Gaussian anti-concentration from $H_\ell K_z^{-1}$. The intuition is mostly clear in 1-d, where $I_N+rX^N$, $X^N$ a Ginibre matrix, has nontrivial least singular value lower bound for any $r\in\mathbb{C}$. At this point one may guess that perhaps the exact structure of $H_\ell$ and $K_z^{-1}$ does not matter, and can one prove a least singular value lower bound for arbitrary coefficient matrices? This idea is stated in the following theorem (see Theorem \ref{theorems6.30} for the non-Gaussian generalization):
\begin{theorem}\label{theorem1290}
Fix integers $m$ and $n$, and $n$ matrices $R_1,\cdots,R_n\in M_m(\mathbb{C})$. Denote by $\mathcal{T}_N$ the following matrix in $M_{mN}(\mathbb{C})$, where $X_1^N,\cdots,X_n^N$ are complex Ginibre matrices:
$$
\mathcal{T}_N:=I_m\otimes I_N+\sum_{\ell=1}^nR_\ell\otimes X_\ell^N
.$$
Then we can find constants $C^R,c^R,c_0^R,c_1^R,N_0(R)$ depending on $n,m,R_1,\cdots,R_n$ such that for all $N\geq N_0(R)$ and any $\epsilon>0$,
 \begin{equation}
\mathbb{P}\{\sigma_{min}(\mathcal{T}_N)\leq\epsilon
\}\leq C^R(N^{c_0^R}\epsilon^{c^R}+\exp(-c_1^RN)).
    \end{equation} 
    Moreover, the two exponents $c_0^R,c^R$ and the constant $c_1^R$ can be chosen to depend only on the dimension $m$ and $n$, but not depend on the specific matrices $R_1,\cdots,R_n$.
\end{theorem}

Theorem \ref{theorem1290} may appear plausibly true although we still  do not know how to handle high dimensional matrix coefficients. The first idea is that we can do an algebraic reduction: we may suppose that $R_1,\cdots,R_n$ have no common nontrivial invariant subspace, which corresponds to assuming that $\mathbb{C}^m$ is a simple module under left $R_1,\cdots,R_n$ action. If this is not true, we can find a composition series, such that each quotient module is a simple left $R_1,\cdots,R_n$-module. See Section \ref{section3.11} for details. In this filtration, the matrix $\mathcal{T}_N$ is directly put into a block upper triangular form, and by simple linear algebra (Fact \ref{fact4570}) we know that we only need to get a least singular value lower bound for each diagonal block where the induced action of $R_1,\cdots,R_n$ acts irreducibly.

To handle each irreducible block, we will use the negative second moment identity (Lemma \ref{blocksecondidentity}) to lower bound the least singular value. The idea of the method is to find the normal vector to all columns except the $a$-th big column, and test the inner product of this $a$-th column against this normal vector. In our setting, any big column is spanned by coefficients $R_\ell$ multiplied by a Gaussian, therefore denoting $A=(A_1,\cdots,A_N)$ the unit normal to all columns except the $a$-th, we need to test the inner product of $A$ with these coefficients $R_\ell$. For fixed $(R_\ell)$ which are not all zero, the Gram matrix is non-degenerate, so this can be transferred to measuring the size of each component $A_i$ in the linear subspace spanned by $R_1,\cdots,R_n$. Therefore, if the final inner product is small in expectation, this would almost force each component $A_i$ to lie almost entirely in the orthogonal complement of $R_1,\cdots,R_n$ in $M_m(\mathbb{C})$. Let $B$ be the entry-wise projection of $A$ onto the orthogonal complement of $R_1,\cdots,R_n$, this implies that $B$ should approximately solve the following equation
\begin{equation}
    B_k+\sum_{i,\ell}\overline{X_\ell(i,k)}R_\ell^*B_i\to 0,\quad \forall k\neq a.
\end{equation}
We prove in Theorem \ref{theorem3.4} that this is impossible with very high probability.

Suppose that $B_1,\cdots,B_N$ span a subspace $E$, and denote $Y(E):=\operatorname{span}_{\ell=1}^n R_\ell^*E$. Then this means each $B_i$ should approximately lie in $E\cap Y(E)$ yet the Gaussian randomness has effective dimension $\dim Y(E)$. Temporarily ignoring very small coefficients, this means that if we know $\dim Y(E)\geq\dim (E\cap Y(E))+1$, then the Gaussian anti-concentration has one more effective dimension than the net size needed to cover the $B_i's$, and this net argument wins. This should imply that the components of $B$ should not all be nearly orthogonal to $\operatorname{span}(R_1,\cdots,R_n)$, completing the proof of Theorem \ref{theorem3.4}. 

So why do we have $\dim (E\cap Y(E))<\dim Y(E)$? This follows from our algebraic assumption that $R_i$ now acts irreducibly, and we can turn this into irreducibility of the action of $R_i^*$, ruling out the possibility that $Y(E)\subset E$ by irreducibility. This is the content of Proposition \ref{propositionnoinvariance}. In the above we have temporarily regarded inequalities with small error as equalities, and we ignored all possibilities of small but nonzero coefficients. In Section \ref{secitons4} we turn this into a rigorous proof allowing all margins, approximations and nearly degenerate coefficients. This gives a full proof of Theorem \ref{theorem1290}.

To the best of the author's knowledge, the use of simple algebraic ideas such as irreducible modules, composition series and irreducible dual actions in proving least singular value estimates appears to be new. Meanwhile, Theorem \ref{theorem1290} may be useful beyond the current setting of polynomials in Ginibre matrices: the theorem covers what may be called patterned Ginibre matrices of finite type and is further very flexible with respect to taking tensor products. The freedom in the choice of $R_i$ opens many possibilities for investigating many other structured Ginibre matrices via Theorem \ref{theorem1290}, which are not accessible otherwise.

\subsubsection{The non-Gaussian case} Generalizing the Gaussian result to the non-Gaussian setting further demands some hard work. The lack of Carbery-Wright inequality for non-log-concave measures requires us to prove a substitute estimate (Proposition \ref{proposition6.2}), and we develop moment comparisons to transfer the proven Gaussian case analysis to the general case. The free probability part (Lemma \ref{lemmaline1522}) also relies on some recent advances in strong convergence from \cite{brailovskaya2024universality}, see also \cite{van2025strong}.

\subsection{Some future directions}

We list here some future directions that are worth further pursuing, and they can be seen as natural next steps of this work.

\begin{Question}Does the singular value estimate in Theorem \ref{theoremonline89} hold at the point $z=\mathfrak{p}(0),$ so that in particular the matrix polynomial $P^N$ is invertible with high probability for a nonzero polynomial $\mathfrak{p}$ with $\mathfrak{p}(0)=0$? The answer is positive when $\mathfrak{p}$ has degree 2, as proven in \cite{cook2022spectrum}, but higher degrees are not covered.
\end{Question}

\begin{Question}
    When $\xi$ has the distribution of a real Gaussian, so that it has a density on $\mathbb{R}$ but still violates \eqref{densityassumption}, does the same Brown measure convergence result still hold? This is particularly interesting in the case where the polynomial coefficients of $\mathfrak{p}$ can take arbitrary complex numbers. The real bottleneck arises in the proof of Theorem \ref{theorem3.4} where a real Gaussian does not provide enough randomness to beat the entropy of the net.
\end{Question}

\begin{Question}
There are also more ambitious questions. First, can we cover distributions $\xi$ with no density, such as the Bernoulli measure? Already for a single Bernoulli matrix $X^N$, its least singular value estimates rely on elaborate technical tools such as inverse Littlewood-Offord theorem \cite{rudelson2008littlewood} and the obstruction to invertibility is arithmetic. But here for higher degree $\mathfrak{p}$, even with Gaussian entries, the obstruction to invertibility is structural rather than arithmetic, as one can read off from the proof in Section \ref{section33}. If we further consider Bernoulli entries, one would need to overcome both arithmetic and structural obstructions simultaneously. Second, this paper only considers $P^N$ depending on $X_1^N,\cdots,X_n^N$ but not their adjoints, and from a free probability viewpoint (see Section \ref{section5.1}) it is also interesting to study $P^N=\mathfrak{p}(X_1^N,\cdots,X_n^N,(X_1^N)^*,\cdots,(X_n^N)^*)$. The current framework does not cover this model.   
\end{Question}

\subsection{The Brown measure limit}\label{section5.1}This section reviews the definition of the Brown measure limit, while the main contents are referenced from \cite[Section 1.4]{cook2022spectrum}. A more comprehensive introduction to free probability can be found in \cite[Chapter 5]{anderson2010introduction} or \cite[Chapter 11]{mingo2017free}.

We recall that a tracial $W*$-probability space $(\mathcal{A},\tau)$ consists of a von Neumann algebra $\mathcal{A}$ and a tracial faithful normal state $\tau$ on $\mathcal{A}$. By Voiculescu's result \cite{voiculescu1991limit}, we have the convergence of the non-commutative distribution of $(X_1^N,\cdots,X_n^N)$ in $*$-moments to the non-commutative distribution for $n$ free circular elements $(c_1,\cdots,c_n)\in\mathcal{A}$. The convergence is in the sense that for any polynomial $\mathfrak{p}\in\mathbb{C}\langle x_1,\cdots,x_n,x_1^*,\cdots,x_n^*\rangle$ we have
$$
\lim_{N\to\infty}\frac{1}{N}\operatorname{Tr}(\mathfrak{p}(X_1^N,\cdots,X_n^N,(X_1^N)^*,\cdots,(X_n^N)^*))=\tau(\mathfrak{p}(c_1,\cdots,c_n,c_1^*,\cdots,c_n^*)).
$$
where the right-hand side is a linear map on the set of polynomials, and is uniquely defined by its values on monomials. For any monomial where we choose $i_1,\cdots,i_k\in[1,n]$ and any $\epsilon_1,\cdots,\epsilon_k\in\{1,*\}$, $\tau(c_{i_1}^{\epsilon_1}\cdots c_{i_k}^{\epsilon_k})$ equals the number of non-crossing pair partitions of $\{1,\cdots,k\}$ where each block $b=(b_1,b_2)$ satisfies that $i_{b_1}=i_{b_2}$ and furthermore $(\epsilon_{b_1},\epsilon_{b_2})=(1,*)\text{ or }(*,1)$. Since polynomial functions are dense among the set of continuous functions, using the fact that Ginibre matrices are bounded with exponentially high probability, the latter implies that for any bounded continuous function $f$ and self-adjoint polynomial $\mathfrak{p}$, 
$$
\lim_{N\to\infty}\frac{1}{N}\operatorname{Tr}(f(\mathfrak{p}(X_1^N,\cdots,X_n^N,(X_1^N)^*,\cdots,(X_n^N)^*)))=\tau(f(\mathfrak{p}(c_1,\cdots,c_n,c_1^*,\cdots,c_n^*)))\quad a.s.
$$For self-adjoint $\mathfrak{p}$, this immediately guarantees convergence of empirical spectral law of $\mathfrak{p}(X_1^N,\cdots,X_n^N,(X_1^N)^*,\cdots,(X_n^N)^*)$ toward the distribution of $\mathfrak{p}(c_1,\cdots,c_n,c_1^*,\cdots,c_n^*)$. For non-self-adjoint $\mathfrak{p}$, the convergence is no longer justified so easily. Girko's idea \cite{girko1985circular} of integration by parts via Green's formula suggests that the empirical measure for eigenvalues of $\mathfrak{p}(X_1^N,\cdots,X_n^N,(X_1^N)^*,\cdots,(X_n^N)^*)$ should converge to the Brown measure defined below, which is given for any twice continuously differentiable function $\psi$ with bounded support on $\mathbb{C}$ via 
\begin{equation}\label{righthandsideof}
\int_\mathbb{C}\psi(\lambda)d\nu_{\mathfrak{p}(c_1,\cdots,c_n,c_1^*,\cdots,c_n^*)}(\lambda):=\frac{1}{4\pi}\int \Delta\psi(z)\tau(\log|z-\mathfrak{p}(c_1,\cdots,c_n,c_1^*,\cdots,c_n^*)|^2)dz.
\end{equation}
To see why this limit is anticipated, note that for any twice continuously differentiable function $\psi$ with bounded support and complex numbers $\lambda_i,1\leq i\leq N$, integration by parts yields
\begin{equation}
    \frac{1}{N}\sum_{i=1}^N\psi(\lambda_i)=\frac{1}{2\pi}\int\Delta\psi(z)\frac{1}{N}\sum_{i=1}^N\log|z-\lambda_i|dz,
\end{equation}and $dz$ is integration over Lebesgue measure on $\mathbb{C}$. Taking the $\lambda_i,1\leq i\leq N$ as eigenvalues of $\mathfrak{p}(X_1^N,\cdots,X_n^N,(X_1^N)^*,\cdots,(X_n^N)^*)$ the above implies that
$$
\frac{1}{N}\sum_{i=1}^N\psi(\lambda_i)=\frac{1}{4\pi}\int\Delta\psi(z)\frac{1}{N}\operatorname{Tr}\left(\log|z-\mathfrak{p}(X_1^N,\cdots,X_n^N,(X_1^N)^*,\cdots,(X_n^N)^*)|^2\right)dz.
$$Observe that the function $|z-\mathfrak{p}(X_1^N,\cdots,X_n^N,(X_1^N)^*,\cdots,(X_n^N)^*)|^2$ is self-adjoint in the entries, and ignoring the singularity and unboundedness of the logarithm we can heuristically take the limit and obtain the right-hand side of \eqref{righthandsideof}.

In this paper we restrict ourselves to holomorphic polynomials $\mathfrak{p}$ so that $\mathfrak{p}$ only depends on $X_1^N,\cdots,X_n^N$ but not their adjoints. Theorem \ref{mainconvergencetheorem} proves that, for this class of holomorphic polynomials $\mathfrak{p}$ of any degree, the spectrum of $\mathfrak{p}(X_1^N,\cdots,X_n^N)$ converges to the Brown measure when evaluated at complex Ginibre matrices.

\subsection{Notation and list of symbols}
Throughout the paper, $M_m(\mathbb{C})$ denotes the linear space of complex-valued square matrices of dimension $m$.  More generally, $M_{m_1,m_2}(\mathbb{C})$ consists of complex matrices of dimension $m_1\times m_2$. The latter space is equipped with the Hilbert-Schmidt norm: for each $A\in M_{m_1,m_2}(\mathbb{C})$, 
$$
\|A\|_{\text{HS}}=\left(\sum_{s=1}^{m_1}\sum_{t=1}^{m_2}|A_{s,t}|^2\right)^\frac{1}{2}.
$$
Then $M_m(\mathbb{C})$ is a Hilbert space with respect to the inner product $\langle A,B\rangle=\operatorname{Tr}(AB^*)$.

Let $\mathcal{H}$ be a complex Hilbert space of complex dimension $d$. Then the
complex Grassmannian $\operatorname{Gr}_\mathbb{C}(k,\mathcal{H})$ consists of linear subspaces of $\mathcal{H}$ of complex dimension $k$. The Grassmannian has real dimension $2k(\dim \mathcal{H}-k)$ and is equipped with the following projection metric 
    $$
\rho(E,E'):=\|P_E-P_{E'}\|_{\text{op}},
    $$where $P_E,P_{E'}$ are orthogonal projections from $\mathcal{H}$ onto $E$ and $E'$ respectively. We can find a $\delta$-net $\mathcal{G}_\delta$ in this Grassmannian with cardinality
    $$
|\mathcal{G}_\delta|\leq(\frac{C}{\delta})^{2k(d-k)}
    $$for some universal constant $C>0$.

\section{First reductions}
\label{section2}
This section defines a linearization matrix for $P^N-zI$, and reduces the least singular value proof of Theorem \ref{theoremonline89} to the proof of Theorem \ref{theorem1290}. Linearization of matrix polynomials has seen many applications in random matrix theory and computations, see for instance \cite{helton2018applications}, \cite{mackey2006vector}.

\subsection{Tensor train decomposition of the polynomial}

We introduce a linearization for the polynomial $\mathfrak{p}$. Indeed, there are many different ways to find such a linearization, and our main proof mechanisms are not specific to one particular linearization matrix. We use the tensor train decomposition introduced in \cite{oseledets2011tensor}.

First write $\mathfrak{p}$ in the following form:
$$
\mathfrak{p}(x)=a_\emptyset+\sum_{\ell=1}^n b_\ell x_\ell+\sum_{q=2}^d\mathfrak{p}_q(x),
$$where $d$ is the maximal degree of $\mathfrak{p}$ and $\mathfrak{p}_q(x)$ is the homogeneous degree-q part of $\mathfrak{p}(x)$.
We can further write each $\mathfrak{p}_q(x)$ as
$$
\mathfrak{p}_q(x)=\sum_{i_1,\cdots,i_q}a_{i_1,\cdots,i_q}^{(q)}x_{i_1}\cdots x_{i_q} 
.$$

For each nonzero $\mathfrak{p}_q$, we choose a minimal Tensor Train/Hankel factorization as follows. Specifically, define the Tensor Train ranks 
$$
r_{q,s}:=\operatorname{rank}[a^{(q)}_{i_1\cdots i_s;i_{s+1}\cdots i_q}],\quad s=1,\cdots,q-1.
$$That is, we flatten the degree-$q$ coefficient tensor across the split 
$$
(i_1,\cdots,i_s)\mid (i_{s+1},\cdots,i_q).
$$
Then we define the intermediate state spaces
$$
E_{q,s}\eqsim \mathbb{C}^{r_{q,s}},\quad s=1,\cdots,q-1. 
$$ These homogeneous components share the same scalar root space $E_\star=\mathbb{C}$. Then for this polynomial $\mathfrak{p}$, we consider the following bouquet linearization space 
$$E=E(\mathfrak{p})=E_\star\oplus \oplus_{2\leq q\leq d,\mathfrak{p}_q\neq 0}
\oplus_{s=1}^{q-1}E_{q,s}
.$$
The subspace $E$ has dimension 
$$
m=m(\mathfrak{p})=1+\sum_{2\leq q\leq d:\mathfrak{p}_q\neq 0}\sum_{s=1}^{q-1}r_{q,s}.
$$
For this degree-q part $\mathfrak{p}_q$ we define a family of maps $G_{q,1}(\cdot),\cdots,G_{q,q}(\cdot)$ such that the argument of each function takes values in $[1,n]$ and 
$$
G_{q,s}(i):E_{q,s}\to E_{q,s-1},\quad 1\leq s\leq q
$$ is a linear map for each $1\leq i\leq n$, with $E_{q,0}=E_{q,q}=\mathbb{C}$, and the coefficient tensor can be written in the following tensor train form:
\begin{equation}\label{tensortrainforms}
a_{i_1\cdots i_q}^{(q)}=G_{q,1}(i_1)G_{q,2}(i_2)\cdots G_{q,q}(i_q)
.\end{equation}

The maps $G_{q,1}(\cdot),\cdots,G_{q,q}(\cdot)$ can be constructed as follows (the construction is not unique).

Start with the flattening
$$
M_1(i_1;i_2,\cdots,i_q)=a^{(q)}_{i_1\cdots i_q}.
$$Then factor it with rank $r_{q,1}$: 
$$
M_1=U_1V_1.
$$Then define $$\operatorname{G_{q,1}}(i_1):=\operatorname{row } i_1\text{ of }U_1$$ so that $G_{q,1}(i_1)\in M_{1,r_{q,1}}$. The remaining tensor is $V_1(\alpha_1;i_2,\cdots,i_q)$. Next we flatten $V_1$ via 
$$
(\alpha_1,i_2)\mid (i_3,\cdots,i_q).
$$Factor it with rank $r_{q,2}$ so we write $$V_1(\alpha_1,i_2;i_3,\cdots,i_q)=U_2V_2
$$
and define 
$$
G_{q,2}(i_2)_{\alpha_1}:=\text{ the $(\alpha_1,i_2)$-th row of }U_2
$$so that $G_{q,2}(i_2)\in M_{r_{q,1},r_{q,2}}$. Proceed in this way, for each $1\leq s\leq q-2$ and each $V_s(\alpha_s,i_{s+1},\cdots,i_q)$ we write 
$$
V_s(\alpha_s,i_{s+1};i_{s+2},\cdots,i_q)=U_{s+1}V_{s+1}
$$ and take $G_{q,s+1}(i_{s+1})_{\alpha_s}$ as the $(\alpha_s,i_{s+1})$-th row of $U_{s+1}$. And finally we define 
$$
G_{q,q}(i_q)_{\alpha_{q-1},1}=V_{q-1}(\alpha_{q-1};i_q).
$$
 This defines $G_{q,3},\cdots,G_{q,q}$.

Then equation \eqref{tensortrainforms} is directly fulfilled by the above definitions. This can be verified inductively:
$$\begin{aligned}
a_{i_1\cdots i_q}^{(q)}&=\sum_{\alpha_1=1}^{r_{q,1}}U_1(i_1,\alpha_1)V_1(\alpha_1,i_2,\cdots,i_q)\\&=\sum_{\alpha_1=1}^{r_{q,1}}G_{q,1}(i_1)\sum_{\alpha_2=1}^{r_{q,2}}U_2(\alpha_1,i_2;\alpha_2)V_2(\alpha_2,i_3,\cdots,i_q)=\cdots. \end{aligned}
$$

\subsection{Linearization step and generated subalgebra}\label{215365}

We will define the linearization matrix in a specific way tailor-made for our goal. For a compact $z$-region, say $|z|\leq R$, consider the root scalar $\theta_*(z)=a_\emptyset-z$.
Then basically any choice of nonzero constants $$\eta_{q,s}\in\mathbb{C}\setminus\{0\},\quad 2\leq q\leq d,1\leq s\leq q-1$$
can generate a valid linearization. For future applications, when $z\neq a_\emptyset$, we fix one choice
\begin{equation}\label{equalitys}
\eta_{q,s}=a_\emptyset -z,\quad 2\leq q\leq d,1\leq s\leq q-1,
\end{equation}
and when $z=a_\emptyset$ we simply take $\eta_{q,s}=1$ for all $q,s$.

Denote by $d_\mathfrak{p}:=\deg\mathfrak{p}$ the maximal degree of the polynomial $\mathfrak{p}$.
Then we define a  deterministic diagonal matrix $K_z$ via 
$$
K_z=(a_\emptyset-z)P_\star+\sum_{2\leq q\leq d_\mathfrak{p},\mathfrak{p}_q\neq 0}\sum_{s=1}^{q-1}\eta_{q,s}P_{q,s},
$$
where $P_\star$ and $P_{q,s}$ are orthogonal projections onto $E_\star$ and $E_{q,s}$. For $z\neq a_\emptyset$, by our definition,
$$
K_z=(a_\emptyset-z)I_E.
$$
so $K_z$ acts as a scalar multiple on the whole space $E$.

Next, we define the edge matrices $H_\ell$.
For each letter $\ell\in\{1,\cdots,n\}$, we define $H_\ell\in M_m(\mathbb{C})$ via the following steps. 

First include the linear part of $\mathfrak{p}$ by setting 
$$
(H_\ell)_{\star,\star}=b_\ell.
$$
Then for each homogeneous chain $\mathfrak{p}_q,q\geq 2$ we define the following chain scaling constant 
$$
\mu_q:=(-1)^{q-1}\prod_{s=1}^{q-1}\eta_{q,s}.
$$Then we enter the following blocks in $H_\ell$:
$$(H_\ell)_{E_\star,E_{q,1}}=\mu_q G_{q,1}(\ell),
$$

$$
(H_\ell)_{E_{q,s-1},E_{q,s}}=G_{q,s}(\ell),\quad 2\leq s\leq q-1,
$$and finally set 
$$
(H_\ell)_{E_{q,q-1},E_\star}=G_{q,q}(\ell),
$$as we identify $E_*\simeq E_{q,q}\simeq \mathbb{C}$.
We finally set all other blocks to be zero except those already defined. Then we define the following linearization

\begin{equation}\label{equation325}
\mathcal{L}_z(X)=K_z\otimes I_N+\sum_{\ell=1}^n
H_\ell\otimes X_\ell.\end{equation}

We now check that this is the correct Schur complement formula:

\begin{lemma}\label{lemma2.222}
The following formula for inverse matrices holds, whenever the inverses exist: 
$$
[(\mathcal{L}_z(X))^{-1}]_{[1,1]}=(P^N(X)-zI_N)^{-1},
$$where $[\cdot ]_{1,1}$ denotes the rows and columns labeled by the root coordinates $E_\star\otimes \mathbb{C}^N\equiv\mathbb{C}^N$.
\end{lemma}

\begin{proof}
    We first consider the simpler case of a single degree-$q$ chain. Write $${G}_{q,s}(X):=\sum_{\ell=1}^n G_{q,s}(\ell)\otimes X_\ell.$$
Then working inside the degree-$q$ auxiliary chain the auxiliary block is 
\begin{equation}\label{schurcomplements}
D_q(X)=\begin{bmatrix}
    \eta_{q,1}I&G_{q,2}(X)&0&\cdots&0\\0&\eta_{q,2}I&G_{q,3}(X)&\cdots&0\\\ddots&\ddots&\ddots&\ddots&\ddots\\0&\cdots&0&\eta_{q,q-2}I&G_{q,q-1}(X)\\0&\cdots&0&0&\eta_{q,q-1}I
\end{bmatrix}.
\end{equation}
The root-to-chain block is given by
$$
B_q(X)=(\mu_q {G}_{q,1}(X),0,\cdots,0)
$$

and the chain-to-root block is given by 
$$
C_q(X)=(0,\cdots 0,G_{q,q}(X))^T.
$$
The matrix $D_q(X)$ is upper triangular with invertible diagonal blocks, so that its $(1,q-1)$-block is 
$$
(D_q(X)^{-1})_{1,q-1}=(-1)^{q-2}\frac{G_{q,2}(X),\cdots G_{q,q-1}(X)}{\prod_{s=1}^{q-1}\eta_{q,s}}
.$$
Therefore, 
$$
B_qD_q^{-1}C_q=\mu_q(-1)^{q-2}\frac{G_{q,1}(X)\cdots G_{q,q}(X)}{\prod_{s=1}^{q-1}\eta_{q,s}}.
$$By definition of $\mu_q$, this gives 
    $$
B_qD_q^{-1}C_q=-G_{q,1}(X)\cdots G_{q,q}(X).
    $$ By the Schur complement formula, we have 
    $$
[(\mathcal{L}_z(X))^{-1}]_{1,1}=(a_\emptyset -z+G_{q,1}(X)\cdots G_{q,q}(X))^{-1}=(P_q^N(X)-z)^{-1}, 
    $$where $P_q^N(X)$ is the degree-$q$ part of $P^N(X)$ plus the constant term.
    This verifies the claim when $\mathfrak{p}$ is homogeneous of degree $q$. 

For the general case, note that the full auxiliary block is block diagonal over all these homogeneous degree$-q$ chains with degree $q\geq 2$, and exactly the same computation shows that the contributions from different degrees sum up to the Schur complement. Therefore, the Schur complement at the root is precisely 
$$
(a_\emptyset-z)I_N+\sum_{\ell=1}^n b_\ell X_\ell+\sum_{2\leq q\leq d,\mathfrak{p}_q\neq 0}G_{q,1}(X)\cdots G_{q,q}(X),
$$
which is exactly $P^N(X)-zI_N$.\end{proof}

We illustrate our construction for polynomials of degree 3:

\begin{example}[Degree-three chain]
Suppose
\[
        \mathfrak{p}(x)=a_\emptyset+\mathfrak{p}_3(x),\qquad
        \mathfrak{p}_3(x)=\sum_{a,b,c=1}^n t_{abc}x_ax_bx_c .
\]
Let
\[
        E_0:=E_\star=\mathbb C,\qquad
        E_1\simeq \mathbb C^{r_1},\qquad
        E_2\simeq \mathbb C^{r_2},
\]
where \(r_1,r_2\) are the tensor-train ranks of the coefficient tensor
\((t_{abc})\). Thus we may write
\[
        t_{abc}=A_aB_bC_c,
\]
where
\[
        A_a:E_1\to E_0,\qquad
        B_b:E_2\to E_1,\qquad
        C_c:E_0\to E_2.
\]
For matrices \(X_1,\ldots,X_n\), define
\[
        A(X):=\sum_{a=1}^n A_a\otimes X_a,\qquad
        B(X):=\sum_{b=1}^n B_b\otimes X_b,\qquad
        C(X):=\sum_{c=1}^n C_c\otimes X_c .
\]
Then
\[
        A(X)B(X)C(X)
        =
        \sum_{a,b,c=1}^n A_aB_bC_c\otimes X_aX_bX_c
        =
        \mathfrak{p}_3(X).
\]

Choose the two nonzero scalars \(\eta_1=\eta_2=a_\emptyset-z\) (when $a_\emptyset\neq z$, or take $\eta_1=\eta_2$=1 when $a_\emptyset=z$) and set
\[
        \mu:=\eta_1\eta_2 .
\]
With respect to the decomposition
\[
        E=E_0\oplus E_1\oplus E_2,
\]
the corresponding linearization is
\[
        \mathcal L_z(X)=
        \begin{pmatrix}
        (a_\emptyset-z)I_N & \mu A(X) & 0\\
        0 & \eta_1 I_{E_1}\otimes I_N & B(X)\\
        C(X) & 0 & \eta_2 I_{E_2}\otimes I_N
        \end{pmatrix}.
\]
Equivalently,
\[
        \mathcal L_z(X)
        =
        K_z\otimes I_N+\sum_{\ell=1}^n H_\ell\otimes X_\ell,
\]
where
\[
        K_z=
        \begin{pmatrix}
        a_\emptyset-z & 0 & 0\\
        0 & \eta_1 I_{E_1} & 0\\
        0 & 0 & \eta_2 I_{E_2}
        \end{pmatrix},
        \qquad
        H_\ell=
        \begin{pmatrix}
        0 & \mu A_\ell & 0\\
        0 & 0 & B_\ell\\
        C_\ell & 0 & 0
        \end{pmatrix}.
\]

The auxiliary block is
\[
        D(X)=
        \begin{pmatrix}
        \eta_1 I_{E_1}\otimes I_N & B(X)\\
        0 & \eta_2 I_{E_2}\otimes I_N
        \end{pmatrix}.
\]
Since \(D(X)\) is upper triangular, its \((1,2)\)-block inverse is
\[
        -(\eta_1\eta_2)^{-1}B(X).
\]
Therefore the Schur complement of the auxiliary block in \(\mathcal L_z(X)\) is
\[
\begin{aligned}
        &(a_\emptyset-z)I_N
        -
        (\mu A(X),0)
        D(X)^{-1}
        \binom{0}{C(X)}        \\
        &\qquad =
        (a_\emptyset-z)I_N
        +\mu(\eta_1\eta_2)^{-1}A(X)B(X)C(X)\\
        &\qquad =
        (a_\emptyset-z)I_N+\mathfrak{p}_3(X)
        =
        \mathfrak{p}(X)-zI_N .
\end{aligned}
\]
In the simplest non-quadratic case, this example exhibits how the cyclic block construction
turns the monomial chain \(A(X)B(X)C(X)\) into the root Schur complement.
\end{example}
\subsection{Operator norm upper bound}

Since the Ginibre matrices $X_i^N$ have bounded operator norm with high probability, we will work on this event throughout the article:

\begin{fact}\label{fact2.7}
    There exists a constant $c_n>0$ such that, with probability at least $1-e^{-c_nN}$, we have $\|X_i^N\|_{op}\leq 8$ for each $1\leq i\leq n$. Under this event, we can find a constant $C_\mathfrak{p}>0$ such that 
    $$
\|\mathfrak{p}(X_1^N,\cdots,X_n^N)\|_{\text{op}}\leq \frac{1}{2}C_\mathfrak{p},
    $$and thus Theorem \ref{theoremonline89} only needs to be proven for all $z\in\mathbb{C}:|z|\leq C_\mathfrak{p}$. Then we can find a constant $K_\mathfrak{p}$ depending only on $\mathfrak{p}$ and the choice of linearization parameters $\eta_{q,s}$ in Section \ref{215365} with $|z|\leq C_\mathfrak{p}$, such that 
    \begin{equation}\label{lines366}
\|K_z\|_{\text{HS}}^2+\sum_{\ell=1}^n\|H_\ell\|_{\text{HS}}^2\leq K_\mathfrak{p},\quad\forall z:|z|\leq C_\mathfrak{p}.\end{equation}
\end{fact}
Meanwhile, we can take the inverse of $K_z$ when $z\neq \mathfrak{p}(0)$:
\begin{fact}\label{fact2.568}
    For all $z\neq a_\emptyset$, we have
    $$
K_z^{-1}=(a_\emptyset-z)^{-1}I_m.
    $$Therefore for $z\neq a_\emptyset$, denote by 
    $$
\mathcal{L}_z'(X)=I_m\otimes I_N+\sum_{\ell=1}^n H_\ell K_z^{-1}\otimes X_\ell,
    $$then for $z\neq a_\emptyset$,
    $$
\sigma_{\text{min}}(\mathcal{L}_z'(X))= |z-a_\emptyset|^{-1}\cdot\sigma_{\text{min}}(\mathcal{L}_z(X)).
    $$
\end{fact}

Then Theorem \ref{theoremonline89} can be deduced from Theorem \ref{theorem1290}:

\begin{proof}[\proofname\ of Theorem \ref{theoremonline89} conditional on Theorem \ref{theorem1290}]
    By Lemma \ref{lemma2.222}, we have 
    $$
\sigma_{\text{min}}(P^N-z)\geq\sigma_{\text{min}}(\mathcal{L}_z(X))
    $$ for the $\mathcal{L}_z(X)$ defined in \eqref{equation325}.
By Fact \ref{fact2.568} we also have, for any $z\neq\mathfrak{p}(0)$,
$$
\sigma_{\text{min}}(\mathcal{L}_z(X))=|z-\mathfrak{p}(0)|\cdot \sigma_{\text{min}}(\mathcal{L}_z'(X))
.$$
The matrix $\mathcal{L}_z'(X)$ is precisely of the form stated in Theorem \ref{theorem1290}, so Theorem \ref{theorem1290} offers a least singular value lower bound for $\mathcal{L}_z'(X)$, and thus for $\mathcal{L}_z(X)$ and then for $P^N-z$. Since each $\mathcal{L}_z'(X)$ has the same dimension, the exponent in $\epsilon$ and $N$ in the least singular value bound can be chosen to be uniform over $z$ (and thus depend only on $\mathfrak{p}$) by Theorem \ref{theorem1290}.
\end{proof}

\section{Least singular value for the general tensor product model}\label{section33}
This section is devoted to the proof of Theorem \ref{theorem1290}, a quantitative least singular value lower bound for the following general tensor product matrices 
\begin{equation}
\mathcal{T}_N= I_m\otimes I_N+\sum_{\ell=1}^n R_\ell\otimes  X_\ell,
\end{equation}where $R_1,\cdots,R_n$ are fixed in $M_m(\mathbb{C})$ and $X_1,\cdots,X_n$ are i.i.d. complex Ginibre matrices. 

\subsection{Filtration spanning
and irreducible modules}\label{section3.11} The first step of the proof is to decompose the space $\mathbb{C}^m$ into invariant subspaces for left $R_1,\cdots,R_n$ action, and then to form irreducible quotients for the action. The decompositions stated here will be used throughout this section.

Consider the algebra $$\mathcal{A}:=\operatorname{alg}(I_m,R_1,\cdots,R_n)\subset M_m(\mathbb{C})$$ and we regard $\mathbb{C}^m$ as a left $\mathcal{A}$-module where $\mathcal{A}$ acts from the left by left multiplication. (Note: in general, the algebra $\mathcal{A}$ does not generate the whole matrix algebra $M_m(\mathbb{C})$, so we need a filtration here.)

We fix an $\mathcal{A}$-invariant composition filtration of $\mathbb{C}^m:$

\begin{equation}\label{filtrationgaps}
0=Z_0\subsetneq Z_1\subsetneq\cdots\subsetneq Z_s=\mathbb{C}_m,
\end{equation} 
where each $Z_a$ is invariant for all left $R_\ell$ actions, and each quotient
\begin{equation}\label{dimsib}
S_{I(b)}:=Z_b/Z_{b-1}
\end{equation}
is a simple left $\mathcal{A}$-module. We use the subscript $I(b),1\leq b\leq s$ to stress that this labels irreducible quotient modules instead of columns.
Then we define $S_{I(1)},\cdots,S_{I(s)}$.

A natural question is how we choose the filtration \eqref{filtrationgaps}, as the choice of composition filtration is in general not unique and not canonical. Indeed, the exact choice does not matter: we will choose an arbitrary composition filtration \eqref{filtrationgaps} and fix it throughout the paper. Although the numerical constants in the least singular value estimate will depend on the choice of filtration, any choice of filtration is enough to complete the proof.

Next, we assign a norm on each $S_{I(b)}$ via the Hilbert space quotient norm. Equivalently, we choose an orthogonal splitting 
$$
\mathbb{C}^m=E_1\oplus\cdots E_s,\quad E_a=Z_a\cap Z_{a-1}^\perp.
$$
Under this orthogonal decomposition, each $R_i$ is placed in a block upper triangular form, and each block map is the following induced map 
\begin{equation}\label{inducedactions}
R_\ell^{(I(b))}:S_{I(b)}\to S_{I(b)}.
\end{equation}Therefore, each $\mathcal{T}_N$ is also of block upper triangular form, where its diagonal blocks are
\begin{equation}\label{label526}
\mathcal{T}_{N,I(b)}= I_{S_I(b)}\otimes I_N+\sum_{\ell=1}^n R_\ell^{(I(b))}\otimes X_\ell,\quad 1\leq b\leq s.
\end{equation}
Therefore, we only need to prove least singular value lower bound on each simple factor, and the overall LSV lower bound can be retrieved via the following triangular reconstruction mechanism:
\begin{fact}\label{fact4570}
For any integer $d$, suppose that $A,C$ are square matrices of size at most $d$ and the block matrix $$
T=\begin{bmatrix}
    A&B\\0&C
\end{bmatrix}
$$acts on an orthonormal Hilbert space direct sum. Assume that 
$$
\sigma_{min}(A)\geq\alpha,\quad \sigma_{min}(C)\geq\gamma,\quad \|T\|\leq M,
$$then 
$$
\sigma_{min}(T)\geq \frac{\alpha\gamma}{3dM}.
$$\end{fact}\begin{proof}Since the inverse matrix has the form
$$
T^{-1}=\begin{bmatrix}
    A^{-1}&-A^{-1}BC^{-1}\\0&C^{-1}
\end{bmatrix},
$$we get that 
$$
\|T^{-1}\|\leq d (\alpha^{-1}+\gamma^{-1}+\alpha^{-1}M\gamma^{-1})\leq d\frac{\alpha+\gamma+M}{\alpha\gamma}\leq d\frac{3M}{\alpha\gamma},
$$ since we have $\alpha,\gamma\leq M$ by definition.
\end{proof} 

\subsection{Consequence of a simple module}
In this algebraic subsection, we show why passing to irreducible $\mathcal{A}$-modules $Z$ is particularly useful for our goal. The main idea is that if $Z$ is irreducible for the left multiplication of $R_i$, then a subset of $\operatorname{End}(Z)$ will also be irreducible for the left multiplication of dual atoms $R_1^*,\cdots,R_n^*$. This will be very useful when we characterize the structure of a normal vector to column subspaces.

We again use a fully abstract formulation. Let $Z$ be a fixed finite-dimensional Hilbert space with given coefficient maps 
$$R_1,\cdots,R_n\in\operatorname{End}(Z).$$ Denote the linear subspace they generate by
$$
\Sigma_R:=\operatorname{span}\{R_1,\cdots,R_n\}\subset\operatorname{End}(Z)
$$and also denote by $\mathcal{D}_R$ its orthogonal complement in the Hilbert-Schmidt norm:
$$
\mathcal{D}_R:=\Sigma_R^\perp\subset\operatorname{End}(Z).
$$
Then for any nonzero linear subspace
$$0\neq E\subset \mathcal{D}_R,$$ we define the subspace generated by $R_\ell^*$ from $E$ by 
\begin{equation}\label{yyyeee}Y(E):=\operatorname{span}_{\ell=1}^n R_\ell^*E
.\end{equation}
Note that the case $E=\mathcal{D}_R$ is also contained in this definition.

\begin{Proposition}
\label{propositionnoinvariance}
Assume that $Z$ is irreducible for the coefficient family, which means 
$$
\text{there does not exist any }0\neq W\subsetneq Z\text{ such that }R_\ell W\subset W\quad \forall \ell.
$$ Then except in the following degenerate case where $$\Sigma_R=0,\quad \mathcal{D}_R=\operatorname{End}(Z),\quad Y(E)=0,$$ the following holds:
$$
0\neq E\subset \mathcal{D}_R\Rightarrow Y(E)\nsubseteq E.
$$ Moreover, $E$ cannot be a left ideal of $\operatorname{End}(Z)$ and we also have
$$
\dim Y(E)-\dim (E\cap Y(E))\geq1.
$$
\end{Proposition}

\begin{proof}
    Suppose, for contradiction, that we have 
    $$
Y(E)\subset E.
    $$Then this means 
    $$
R_\ell^*E\subset E,\quad \forall \ell.
    $$Since $Z$ is irreducible for the coefficient family, Burnside's theorem implies that 
    $$
\operatorname{alg}(I_,R_1,\cdots,R_n)=\operatorname{End}(Z).
    $$Taking adjoints of each element, we get that 
    $$
\operatorname{alg}(I_,R_1^*,\cdots,R_n^*)=\operatorname{End}(Z).$$
Then $E$ is invariant under the left action by any element of $\operatorname{End}(Z)$. In other words, $E$ is a nonzero left ideal of $
\operatorname{End}(Z)$.

We know by a standard fact (see Lemma \ref{howdoweuse681}) that each nonzero left ideal has the form 
$$
E=\operatorname{End}(Z)P
$$ for some nonzero projection onto a subspace $\operatorname{Ran} P\neq 0.$
But since 
$$
E\subset \mathcal{D}_R=\Sigma_R^\perp,
$$we get that 
$$
\langle AP,R_\ell\rangle_{\text{HS}}=0\quad\forall A\in\operatorname{End}(Z),\forall\ell.
$$Equivalently, we have
$$ \operatorname{Tr}((AP)^*R_\ell)=0\quad\forall A,  $$
which gives 
$$
\operatorname{Tr}(A^*R_\ell P)=0\quad\forall A.
$$Since $A$ is arbitrary, this means that 
$$R_\ell P=0\quad \forall \ell.
$$Therefore we get 
$$
R_\ell\cdot \operatorname{Ran}(P)=0\quad\forall\ell.
$$Then $\operatorname{Ran }(P)$ is a common invariant subspace for the $R_\ell$'s. If $0\neq \operatorname{Ran}(P)\subsetneq Z$, this is a contradiction to irreducibility.

Finally, if $\operatorname{Ran}(P)=Z$, then $P=I$ and $R_\ell=R_\ell P=0\quad \forall \ell.$ This is the excluded case.
\end{proof}

\subsection{Least singular value for dual operators with a dimensional surplus}We introduce a crucial least singular value estimate which is a consequence of Proposition \ref{propositionnoinvariance}.

 Let $Z$ be an irreducible left $\mathcal{A}$-module, and recall that $\mathcal{D}_R\subset \operatorname{End}(Z)$ is the orthogonal subspace to $\Sigma_R=\operatorname{span}(R_1,\cdots,R_n)$. Let $a$ be any index where $1\leq a\leq N$, and define the following random operator 
$$
\mathfrak{D}_a:\mathcal{D}_R^N\to\operatorname{End}(Z)^{N-1}
$$via 
\begin{equation}
    (\mathfrak{D}_aB)_k=B_k+\sum_{i,\ell}\overline{X_\ell(i,k)}R_\ell^*B_i,\quad k\neq a.
\end{equation} Here we simply use $R_\ell$ to denote the induced action of $R_\ell$ on $Z$ (see \eqref{inducedactions} for the more precise notation of induced action, but we omit the induced subscript $I(b)$ here for simplicity), and $R_\ell^*$ is simply the adjoint matrix of this induced action $R_\ell$.

Indeed, by definition the law of $\mathfrak{D}_a$ does not depend on the choice of $a$. The mapping $\mathfrak{D}_a$ arises as we consider the normal vector to the span of all columns of $\mathcal{T}_{N,I(s)}$ with the $a$-th column removed (for any block $s$), and the input vector $B$ ranges from all normal vectors taking value in $\mathcal{D}_R$. We have the following lower bound for the least singular value of $\mathfrak{D}_a$:
\begin{theorem}\label{theorem3.4}
For any such simple $\mathcal{A}$-module $Z$ where $R_1,\cdots,R_n$ do not act trivially and $\mathcal{D}_R\neq 0$, we can find an integer $D_Z>0$ depending only on $\dim_\mathbb{C}Z$  such that whenever $N\geq N_0(Z,R_1,\cdots,R_n)$ (where $N_0(Z,R_1,\cdots,R_n)$ is a fixed constant depending only on parameters of the module), we have
    $$\mathbb{P}(
\sigma_{\text{min}}(\mathfrak{D}_a)\leq N^{-D_Z})\leq\exp(-\Omega(N)).
    $$The multiplicative constant in $\Omega(N)$ depends only on $n$ and not on $R_1,\cdots,R_n$.
\end{theorem}

Here $R_1,\cdots,R_n$ act trivially means that they act by zero on $Z$ and in particular $Z$ is one dimensional. Except this degenerate case,
we always have $\dim \Sigma_R>0$ so that $\dim \mathcal{D}_R<\dim \operatorname{End}(Z)$. Thus the output dimension of $\mathfrak{D}_a$ is larger than the input dimension when $N$ is large enough, and therefore a good least singular value lower bound is expected to hold for $\mathfrak{D}_a$. The proof of Theorem \ref{theorem3.4} is deferred to Section \ref{secitons4}. The proof is highly technical and is one of the main contributions of this paper. We now complete the rest assuming Theorem \ref{theorem3.4} holds.

\subsection{Proving square LSV for a single factor via exposed columns}

To bound the least singular value of $\mathcal{T}_N$ which is a square matrix, we will use the following negative second moment identity. The proof is presented in Appendix \ref{AppendixA}.

\begin{lemma}\label{blocksecondidentity}
    Let $Z$ be a Hilbert space, and we write $Z^N=E_1\oplus\cdots\oplus E_N$ where each $E_i$ is a copy of $Z$. Let $T:Z^N\to Z^N$ be an injective linear map. Denote the block columns as $T_a:=T|_{E_a}:E_a\to Z^N$. For each $a\in[N]$, denote by 
    $$
E_{-a}:=\oplus_{b\neq a}E_b,\quad T_{-a}=T_{E_{-a}}.
    $$Also denote by $P_a:Z^N\to \operatorname{Ran}(T_{-a})^\perp=\ker T_{-a}^*$ the orthogonal projection. Then we define the exposed block maps 
    $$
\mathfrak{Exp}_a:=P_aT_a:E_a\to\operatorname{Ran}(T_{-a})^\perp.
    $$Then we have, denoting by $\sigma_1(\cdot),\sigma_2(\cdot),\cdots$ the respective singular values of $T$ and $\mathfrak{Exp}_a$,
    $$
\sum_{s=1}^{N\dim Z }\sigma_s(T)^{-2}=\sum_{a=1}^N\sum_{r=1}^{\dim Z}\sigma_r(\mathfrak{Exp}_a)^{-2}.
    $$
\end{lemma}
We now apply Lemma \ref{blocksecondidentity} to the matrix $\mathcal{T}_N$. We fix an exposed block column $a$ (which means all the columns indexed by $\{1,\cdots,m\}\otimes e_a$ where $e_a$ is the $a$-th column for an $N$-column matrix). After removing the $a$-th block column, denote by $\mathcal{T}_{-a}$ the resulting matrix. Denote by 
$$
W_a:=\ker (\mathcal{T}_{-a}^*)\subset Z^N.
$$Then almost surely, $\dim W_a=\dim Z.$ Indeed, the following fact is proven in Appendix \ref{AppendixA}:
\begin{fact}\label{facts3.33}
    Almost surely, $\dim W_a=\dim Z$ and almost surely, $\mathcal{T}_N$ is invertible.
\end{fact}

The $a$-th exposed column of $\mathcal{T}_N$ has the following form 
$$
C_a(\xi):Z\to Z^N,
$$
$$
C_a(\xi)u=e_au+\sum_{i=1}^N\sum_{\ell=1}^n\xi_i^\ell e_iR_\ell u,
$$where $e_1,\cdots,e_N$ are standard basis vectors and where $\xi_i^\ell$ are i.i.d. copies of complex Gaussian random variables with variance $\frac{1}{N}$. Then we project onto $W_a$ and denote by 
$$
D_a(\xi):=P_{W_a}C_a(\xi):Z\to W_a.
$$

Our goal is to prove that, with very high probability over the non-exposed columns in $\mathcal{T}_{-a}$, we can find some fixed constant $B_Z>0$ such that
\begin{equation}\label{testvectors}
\|\det D_a(\xi)\|_{L_\xi^2}\geq N^{-B_Z}.
\end{equation}
This is because a lower bound on $\det D_a(\xi)$ can easily yield a lower bound on $\sigma_{\text{min}}(D_a(\xi))$ which we then feed in Lemma \ref{blocksecondidentity}, since it has a fixed size and bounded operator norm.

\subsection{Exposed coefficient frame lower bound}To prove \eqref{testvectors}, we regard $D_a(\xi)$ as a Gaussian vector (conditional on $\mathcal{T}_{-a}$) and compute its variance by testing it against other test vectors. We take any unit test vector 
$$A:Z\to W_a\subset Z^N.
$$We write the coordinates of $A$ in terms of the coordinate structure on $Z^N$ as
$$
A=(A_i)_{i=1}^N,\quad A_i\in\operatorname{End}(Z).
$$Then we can write the inner product as 
$$
\langle A,D_a(\xi)\rangle=\langle A_a,I_Z\rangle+\sum_{i,\ell}\xi_i^\ell\langle A_i,R_\ell\rangle.
$$Then $\mathbb{E}[\langle A,D_a(\xi)\rangle]=\langle A_a,I_Z\rangle.$

At this point, we divide into two different situations:
\begin{enumerate}
    \item Either on $Z$, all the actions $R_\ell$ act by 0. Then $Z$ is simply one dimensional and the matrix $\mathcal{T}_N$ is simply the $N$-dimensional identity matrix. Then its least singular value lower bound is immediate.
    \item The left actions $R_\ell$ do not act by 0 on $Z$. We will then use anti-concentration of the Gaussian vectors. The rest of the proof will focus only on this situation.
\end{enumerate}

We prove the following characterization of the structure of $A$:

\begin{Proposition}\label{proposition}
There exist some $B_Z>0$ depending only on $\dim Z$ (where $Z$ is an irreducible factor), and a constant $N_0(R)\in\mathbb{N}_+$ depending on the nontrivial $R_1,\cdots,R_n$ action, such that whenever $N\geq N_0$, with probability at least $1-\exp(-\Omega(N))$ over the non-exposed columns in $\mathcal{T}_{-a}$, 
we have
\begin{equation}\label{equationest702}
\frac{1}{N}\sum_{i=1}^N\sum_{\ell=1}^n |\langle A_i,R_\ell\rangle|^2\geq N^{-B_Z},
\end{equation}
for any unit vector $A:Z\to W_a$. Here, the coefficient in $\Omega(N)$ does not depend on $R$.
\end{Proposition}

\begin{proof} 
We let the value of $B_Z$ be determined at the end of the proof. 
Suppose that there exists a unit vector $A:Z\to W_a$ such that the lower bound \eqref{equationest702} does not hold. Since 
$$A(Z)\subset W_a=\operatorname{ker}(\mathcal{T}_{-a}^*), 
$$the block components of $A$ should satisfy the following adjoint equations:
\begin{equation}\label{adjointequation}A_k+\sum_{i,\ell}\overline{X_\ell(i,k)}R_\ell^*A_i=0
,\quad k\neq a.
\end{equation}

Recall that we defined the fixed linear subspace 
$$
\Sigma_{R}=\operatorname{span}\{R_1,\cdots,R_n\}\subset\operatorname{End}(Z).
$$
Since $\Sigma_{R}$ is fixed (independent of $N$) and nonzero by assumption, the following function
$$
B\to(\langle B,R_1\rangle,\cdots,\langle B,R_n\rangle)\in \mathbb{C}^n,\quad B\in\Sigma_R
$$ has a positive lower singular value $\mathfrak{Sing}_Z$ on $\Sigma_{R}$ (this follows since the above function is nonzero for any nonzero $B$ and is continuous in $B$, and the unit ball of $\Sigma_R$ is compact). The precise lower bound for $\mathfrak{Sing}_Z$ depends on the module $Z$ and actions $R_i$, but what is important here is only that $\mathfrak{Sing}_Z$ is fixed and positive. Therefore, 
$$
\sum_{i=1}^N \|P_{\Sigma_{R}}A_i\|_{\text{HS}}^2\leq\frac{1}{\mathfrak{Sing}_Z^2}\sum_{i=1}^N\sum_{\ell=1}^n|\langle P_{\Sigma_R}A_i,R_\ell\rangle|^2=\frac{1}{\mathfrak{Sing}_Z^2}\sum_{i=1}^N\sum_{\ell=1}^n|\langle A_i,R_\ell\rangle|^2
$$
since $R_\ell\in\Sigma_R$. Then if the unit vector $A$ violates \eqref{equationest702}, this implies that
$$
\sum_{i=1}^N\|P_{\mathcal{D}_R}A_i\|_{\text{HS}}^2=1-\sum_{i=1}^N\|\mathcal{P}_{\Sigma_R}A_i\|_{\text{HS}}^2\geq 1-\frac{1}{\mathfrak{Sing}_Z^2}N^{-B_Z+1}\geq\frac{3}{4}
$$whenever $N$ is sufficiently large, where we recall that 
$\mathcal{D}_R$ is the orthogonal complement of $\Sigma_R$ and we use the fact that $A$ is a unit vector. In particular, $\mathcal{D}_R\neq 0$ and $\Sigma_R\subsetneq \operatorname{End}(Z)$.

Therefore, we denote by
$$
B_i:=P_{\mathcal{D}_R}A_i,\quad C_i=\mathcal{P}_{\Sigma_R}A_i
,$$
then 
\begin{equation}\label{atlines739}
\sum_{i=1}^N\|B_i\|_{\text{HS}}^2\geq\frac{3}{4},\quad\sum_{i=1}^N\|C_i\|_{\text{HS}}^2\leq\frac{1}{\mathfrak{Sing}_Z^2} N^{-B_Z+1}.
\end{equation}On the good event $\Omega_1:=\{\|X_\ell\|_{\text{op}}\leq 8\forall \ell\}$, which has probability $1-\exp(-\Omega(N)),$ we have that 
$$
\|\sum_{i,\ell}\overline{X_\ell(i,k)}R_\ell^*C_i\|\leq \mathcal{C}_RN^{(-B_Z+1)/2},
$$for some $\mathcal{C}_R>0$ depending only on the module $Z$ and $R_1,\cdots,R_n$. 

Projecting the adjoint equation \eqref{adjointequation} onto its $\mathcal{D}_R$-component, we get that 
\begin{equation}\label{atlines747}
\|B_k+\sum_{i,\ell} \overline{X_\ell(i,k)}R_\ell^*B_i\|\leq \mathcal{C}_RN^{(-B_Z+1)/2}\quad \forall k\neq a.
\end{equation}
However, by Theorem \ref{theorem3.4}, with probability $1-\exp(-\Omega(N))$, the least singular value of $\mathfrak{D}_a$ is at least $N^{-D_Z}$. Thus if we take $B_Z=2D_Z+3$, then for large $N$ we cannot find any $B$ satisfying both \eqref{atlines747} and \eqref{atlines739}. This leads to a contradiction.
\end{proof}

\subsection{Frame lower bound}
The following is a direct consequence of Proposition \ref{proposition}: 

\begin{corollary}\label{corollary3.70}Let $R_1,\cdots,R_n$ act irreducibly and non-trivially on the module $Z$.
Then with probability $1-\exp(-\Omega(N))$ over the columns in $\mathcal{T}_{-a}$,
the affine Gaussian map 
$$
D_a(\xi):Z\to W_a
$$ satisfies that, for any unit vector $A\in\operatorname{Hom}(Z,W_a)$,
\begin{equation}\label{covariancelowerbound}
\mathbb{E}\left|\langle A,D_a(\xi)-\mathbb{E}[D_a(\xi)]\rangle\right|^2\geq N^{-B_Z}
\end{equation} whenever $N\geq N_0(R)$. Here $N_0(R)\in\mathbb{N}_+$ depends on the module $Z$ and $R_1,\cdots,R_n$, while the constant $B_Z$ depends only on $\dim Z$. The coefficient in $\Omega(N)$ depends only on $n$. Moreover, on the same high probability event over $\mathcal{T}_{-a}$ and when $N\geq N_0$, we can find a constant $C_Z>0$ depending only on $\dim Z$ such that we have the following estimate
\begin{equation}\label{determinantmeans}
\mathbb{E}[|\det D_a(\xi)|^2]\geq C_Z N^{-B_Z\dim Z}.
\end{equation}\end{corollary}

\begin{proof}
The first claim is immediate by Proposition \ref{proposition}. For fixed $\mathcal{T}_{-a}$, we now regard $D_a(\xi)$ as an affine Gaussian matrix. Then \eqref{covariancelowerbound} shows that it can be rewritten as 
\begin{equation}\label{shortcuts}
D_a(\xi)=\mathbb{E}D_a(\xi)+N^{-B_Z/2}G+G_0,
\end{equation}where $G$ is a standard $\dim Z\times \dim Z$ complex Ginibre matrix with each entry of variance 1, and $G_0$ is another Gaussian matrix independent of $G$. We condition on $G_0$, then the polynomial 
$$
\det(G_0+\mathbb{E}D_a(\xi)+N^{-B_Z/2}G)
$$
has, in its highest homogeneous chaos, the term $N^{-B_Z\dim Z/2}\det G$.
Since different Gaussian chaos are orthogonal in $L^2$, we get that
$$
\mathbb{E}_G[|\det(G_0+\mathbb{E}D_a(\xi)+N^{-B_Z/2}G)|^2]\geq N^{-B_Z\dim Z}\mathbb{E}|\det G|^2=(\dim Z)! N^{-B_Z\dim Z}.
$$ This verifies the claimed determinant $L^2$ lower bound.
\end{proof}
\subsection{Carbery-Wright and block negative second moment}

To convert the estimate in expectation in \eqref{determinantmeans} to a small ball estimate, we shall need the following anti-concentration result of Carbery and Wright:

\begin{theorem}\label{theoremlines102}(\cite[Theorem 8]{carbery2001distributional}). There exists a universal constant $C>0$ such that the following holds. Let $f:\mathbb{R}^n\to\mathbb{R}$ be any degree $d$ polynomial and $\mu\sim\mathcal{N}(0,I_n)$ is the standard Gaussian on $\mathbb{R}^n$. Then for every $\epsilon>0$,
$$
\mathbb{P}_{x\sim\mu}\{|f(x)|\leq\epsilon\cdot\mathbb{E}_{x\sim\mu}|f(x)|\}\leq C\cdot d\epsilon^{1/d}.
$$
    
\end{theorem}

Now we can complete the proof of Theorem \ref{theorem1290}.
\begin{proof}[\proofname\ of Theorem \ref{theorem1290} for irreducible factors in \eqref{label526}] We write $Z=S_{I(b)}$ for an irreducible factor under left $R_1,\cdots,R_n$ action. When the action is trivial, then we have identity matrix on this factor $Z$ and the proof is trivial. Otherwise, proceed as follows. To simplify notation, we again use $\mathcal{T}_N$ to denote any of the irreducible diagonal factor $\mathcal{T}_{N,I(b)}$.

We work on the high probability event on $\mathcal{T}_{-a}$ stated in Corollary \ref{corollary3.70}, which holds with probability $1-\exp(-\Omega(N))$. Then for any fixed $P_{W_a}$, each entry of the matrix $D_a(\xi)$ is an affine linear combination of independent complex Gaussians, and therefore $\det D_a(\xi)$ is a polynomial in independent complex Gaussians of bounded degree. While the polynomial differs as $\mathcal{T}_{-a}$ changes, the degree of the polynomial is uniformly bounded.
 Then Theorem \ref{theoremlines102} (applied to the real polynomial $|\det D_a(\xi)|^2$, viewing the complex Gaussian variables through their real and imaginary parts) gives that, for any $\epsilon>0$, we can find constants $C_Z,C_Z',c_Z>0$ such that
\begin{equation}\label{cerberywrights}
\mathbb{P}_\xi\{|\det D_a(\xi)|\leq\epsilon \}\leq C_ZN^{C_Z'}\epsilon^{c_Z}
,\end{equation}
where the constants $C_Z,C_Z',c_Z$ depend only on the module $Z$.

We work on the good operator norm event $\|X_\ell\|\leq 8\quad \forall \ell$, and thus we can find some $C_{2,Z}>0$ so that $\|D_a(\xi)\|\leq C_{2,Z}$.
From this we deduce that 
$$\sigma_{min}(D_a(\xi))\leq \epsilon$$
implies 
$$|\det D_a(\xi)|\leq \epsilon(C_{2,Z})^{\dim Z-1}.
$$
Therefore, we deduce that for any $\epsilon>0$,
\begin{equation}\label{goodlines839}
\mathbb{P}(\sigma_{min}(D_a(\xi))\leq \epsilon)\leq C_{3,Z}N^{C_Z'}\epsilon^{c_Z}
.
\end{equation}
Finally, we use the block negative second moment identity (Lemma  \ref{blocksecondidentity})
to deduce that, if each big column $a$ satisfies 
$$
\sigma_{min}(D_a)\ge t,
$$then 
$$
\sigma_{min}(\mathcal{T}_N)\geq \frac{t}{\sqrt{N\dim Z}}.
$$ We take a union bound for \eqref{goodlines839} over $a\in[N]$, combine with the high probability good event on $\mathcal{T}_{-a}$ and on $\|X_\ell\|$ (both hold with probability $1-\exp(-\Omega(N))$) to conclude that 
\begin{equation}\label{488eachirreducible}
    \mathbb{P}\{\sigma_{min}(\mathcal{T}_N)\leq\epsilon\}\leq C_{4,Z}N^{C_{5,Z}}\epsilon^{c_Z}+\exp(-\Omega(N)),
\end{equation}
where $C_{i,Z},i=2,3,4,5$ are constants depending on the module $Z$.
This completes the proof in the simple factor case.
\end{proof}

 \begin{proof}[\proofname\ of Theorem \ref{theorem1290} for fixed coefficients] We work under the good operator norm event so that $\|\mathcal{T}_N\|\leq C_6$ for some $C_6>0$ depending on $R_1,\cdots,R_n$. We use that each irreducible factor $\mathcal{T}_{N,I(b)},1\leq b\leq s$ individually satisfies an estimate of the form \eqref{488eachirreducible} with possibly different exponents, and then use Fact \ref{fact4570} to combine them into an overall least singular value estimate for $\mathcal{T}_N$ by changing the exponents in $\epsilon$ and in $N$.\end{proof}

\begin{remark}
    Since $D_a(\xi)$ is conditionally Gaussian, we may avoid using Carbery-Wright and directly estimate least singular value via the decomposition \eqref{shortcuts}. But as we later consider non-Gaussian entries, this shortcut is no longer available, yet the present route of proof via Carbery-Wright is more generalizable.
\end{remark} 

\subsubsection{Constant dependence}
\begin{proof}[\proofname\ of Theorem \ref{theorem1290}, the exponents $c_0^R,c^R$ can depend only on $m,n$.]
We note that for any coefficient $R_1,\cdots,R_n$ acting on $\mathbb{C}^m$, when we define a composition series where each quotient is simple, the possible dimension series $\dim S_{I(b)}$ of each irreducible factor \eqref{dimsib} arising in the composition series \eqref{filtrationgaps} has only finitely many possibilities since $m$ is fixed. The total possible number of dimension series is at most the number of ways to write $m$ as the sum of positive integers, and multiplied by $m!$ which is the number of permutations on $m$.

Recall that in Theorem \ref{theorem3.4}, the exponent $D_Z$ in $N$ only depends on $\dim Z$ (the dimension of the irreducible factor), although the constant $N_0$ there depends on $R$. In the proof of Proposition \ref{proposition}, we use the choice $B_Z=2D_Z+3$, so it also depends only on $\dim Z$. The exponents in $\epsilon$ arising from Carbery-Wright also only rely on $\dim Z$, so we can use the same exponent in $\epsilon$ and $N$ for the least singular value lower bound for all coefficients $R_1,\cdots,R_n$ generating the composition series with the same dimension sequence of irreducible factors. Since the number of dimension series is finite, we can choose the worst exponent, i.e. the largest exponent in $N$ and the smallest exponent in $\epsilon$, which covers all possible coefficients. The constant $c_1^R$ comes from an operator norm event on $\|X_\ell\|$, so it depends only on $n$.
\end{proof}

\section{Dual least singular value estimate with dimensional surplus}
\label{secitons4}This section is devoted to the proof of Theorem \ref{theorem3.4}. Recall that we consider the random mapping $\mathfrak{D}_a$ defined via
$$(\mathfrak{D}_aB)_k=B_k+\sum_{i,\ell}\overline{X_\ell(i,k)}R_\ell^*B_i,\quad k\neq a.$$

Let $E\subset \mathcal{D}_R:=\Sigma_R^\perp\subset \operatorname{End}(Z)$ be any linear subspace, and define the following coefficient maps 
\begin{equation}
\Psi_E:E\otimes\mathbb{C}^n\to \operatorname{End}(Z),\quad \Psi_E((x_\ell)_\ell)=\sum_{\ell=1}^n R_\ell^*x_\ell.
\end{equation} Then, by definition in \eqref{yyyeee}, $Y(E)=\operatorname{Ran}\Psi_E$. However, the smallest singular value of this map can become very small and make later small-ball estimates infeasible. To remedy this, take a scale $\sigma\geq 0$ (whose value will be determined later), and we define the following subset of $Y(E)$:

\begin{equation}\label{scaleysigmae}
Y_\sigma(E):=\text{span}\{\text{left singular vectors of }\Psi_E\text{ with singular value }\geq\sigma\}.
\end{equation}We call this the active image, which is the part of $Y(E)$ such that the Gaussian random term has variance at least $\sigma^2/N$. The dimensions in $Y(E)\cap( Y_\sigma(E))^\perp$ are discarded and not counted for small ball dimensions.

Then we define the following quotient map 
\begin{equation}
P_{E,\sigma}:E\to \operatorname{End}(Z)/Y_\sigma(E),\quad P_{E,\sigma}x=[x]
.\end{equation}

Take another threshold $\tau>0$ which will be fixed later, and we define 
\begin{equation}\label{scaletausigmae}
C_{\tau,\sigma}(E)
:=\text{span}\{\text{right singular vectors of }P_{E,\sigma}\text{ with singular value }\leq\tau\}.\end{equation}Informally, $C_{\tau,\sigma}(E)$ is the part of $E$ which is within an angle $\tau$ to the active image $Y_\sigma(E)$.

Throughout the section we work under the following good operator norm event: there exists some constant $M_{\text{op}}>0$ such that 
\begin{equation}\label{goodoperatornorm}
 \Omega_{\text{good}}:=\{   \|\Psi_E\|\leq M_{\text{op}}\quad \forall E,\quad \|X_\ell\|\leq 8\quad \forall 1\leq \ell\leq n\}.
\end{equation}
Since $\|\Psi_E\|\leq\sum_{\ell=1}^n\|R_\ell^*\|$, $M_{\text{op}}$ always exists and we fix such a choice. By standard properties of operator norm bounds for Ginibre, we have $$\mathbb{P}(\Omega_{\text{good}})\geq 1-\exp(-\Omega(N)),$$
where the multiplicative constant in $\Omega(N)$ depends only on $n$.
\subsection{Preliminaries for a net argument}In this section, we list all technical preparations that we will need to prove Theorem \ref{theorem3.4}.

\subsubsection{A robust lemma for dimension surplus}
We shall first need the following quantitative strengthening for Proposition \ref{propositionnoinvariance}:

\begin{Proposition}
\label{proposition1044}    For the fixed module $Z$ and nontrivial actions $R_1,\cdots,R_n$, we can find two constants $\sigma_0\in(0,1)$ and $\tau_0>0$ such that, for any $\sigma\leq\sigma_0$ and $\tau\leq \tau_0$, and any linear subspace $0\neq E\subset \mathcal{D}_R$, we have 
    $$
\dim Y_\sigma(E)\geq\dim C_{\tau,\sigma}(E)+1.
    $$
\end{Proposition}

For two linear subspaces $Y,C\subset F$ in a Hilbert space $F$, and $t>0$, we use the following notation $Y\subset_t C$ to mean that 
\begin{equation}Y\subset_t C\Rightarrow
\operatorname{dist}(y,C)\leq t\|y\|\quad\forall y\in Y.
\end{equation}
We shall need the following simple linear algebra lemma (proof see Appendix \ref{AppendixA})
\begin{lemma}\label{lemma4.20}
    Let $F$ be a finite-dimensional linear subspace and $C,Y\subset F$ with 
    $$
\dim C=\dim Y=r.
    $$We suppose that for some $\tau<1$ we have
    $$
\|P_{Y^\perp}x\|\leq\tau\|x\|\quad\forall x\in C.
    $$Then we have the other side approximate inclusion 
    $$
Y\subset_{C_\tau}C,\quad C_\tau=\frac{\tau}{\sqrt{1-\tau^2}}.
    $$
\end{lemma}

\begin{proof}[\proofname\ of Proposition \ref{proposition1044}]
    
By definition of $C_{\tau,\sigma}(E)$, we have 
$$C_{\tau,\sigma}(E)\subset_{\tau}Y_\sigma(E). 
$$
Then the projection $P_{Y_\sigma(E)}\mid_{C_{\tau,\sigma}(E)}$ of $C_{\tau,\sigma}(E)$ onto $Y_\sigma(E)$ is injective, so that $\dim Y_\sigma(E)\geq\dim C_{\tau,\sigma}(E)$.

If we have $\dim C_{\tau,\sigma}(E)=\dim Y_{\sigma}(E)$, then Lemma \ref{lemma4.20} yields
$$
Y_\sigma(E)\subset_{\frac{\tau}{\sqrt{1-\tau^2}}}C_{\tau,\sigma}(E).
$$Since $C_{\tau,\sigma}(E)\subset E$, we have 
$$
Y_\sigma(E)\subset_{\frac{\tau}{\sqrt{1-\tau^2}}}E.
$$ Now we take into account the inactive image subspace. Take any $x\in E$ and any $\ell$, write 
$$
S_\ell(x)=\Psi_E(x\otimes e_\ell)=\Psi_{\geq\sigma}(x\otimes e_\ell)+\Psi_{<\sigma}(x\otimes e_\ell)
$$where $\Psi_{\geq\sigma}$ and $\Psi_{<\sigma}$ are respectively the parts of $\Psi$ with singular values larger or smaller than $\sigma$. Then by definition, 
$$
\Psi_{\geq\sigma}(x\otimes e_\ell)\in Y_\sigma(E),
$$and as $Y_\sigma(E)\subset \frac{\tau}{\sqrt{1-\tau^2}}E$, we have 
$$
\operatorname{dist}(\Psi_{\geq\sigma}(x\otimes e_\ell),E)\leq \frac{\tau}{\sqrt{1-\tau^2}}\|\Psi_{\geq\sigma}(x\otimes e_\ell)\|.
$$Since $\|\Psi_E\|\leq M_{\text{op}}$, we write 
$$
\|\Psi_{\geq\sigma}(x\otimes e_\ell)\|\leq M_{\text{op}}\|x\|.
$$

The inactive part simply satisfies 
$$
\|\Psi_{<\sigma}(x\otimes e_\ell)\|\leq\sigma\|x\|.
$$Combining both inequalities, we get
$$
\operatorname{dist}(S_\ell x,E)\leq (\sigma+M_{\text{op}}\frac{\tau}{\sqrt{1-\tau^2}})\|x\|.
$$That is, for each $1\leq \ell\leq n$, we have 
\begin{equation}\label{montannastate}
\|P_{E^\perp}S_\ell P_E\|_{op}\leq M_{\text{op}}\frac{\tau}{\sqrt{1-\tau^2}}+\sigma.
\end{equation}

However, we claim that there exists a constant $C_\mathcal{R}>0$ such that for any nonzero subspace $0\neq E\subset \mathcal{D}_R$, we have that 
\begin{equation}
\label{frameequations2}\sum_{\ell=1}^n\|P_{E^\perp}S_\ell P_E\|_{\text{op}}\geq C_\mathcal{R}>0.
\end{equation}
Granted this estimate, we see that \eqref{montannastate} cannot hold for each $1\leq\ell\leq n$ when $n(M_\text{op}\frac{\tau}{\sqrt{1-\tau^2}}+\sigma)< C_\mathcal{R}$.  
This forces, when $\sigma$ and $\tau$ are both small,
$$ 
\dim C_{\tau,\sigma}(E)+1\leq \dim Y_\sigma(E).
$$

Finally, we justify the validity of \eqref{frameequations2}. 
For each dimension $0<d\leq \dim \mathcal{D}_R$, the function $\sum_{\ell=1}^n\|P_{E^\perp}S_\ell P_E\|_{op}$ is a continuous function for $E$ in the Grassmannian $\operatorname{Gr}(d,\mathcal{D}_R)$, and it never takes the value 0 (otherwise if it took the value 0 at some subspace $E$, then $E$ is invariant under the left multiplication of $R_\ell^*$ for all $\ell$, so that $E$ is a left ideal of $\operatorname{End}(Z)$, but this leads to a contradiction to Proposition \ref{propositionnoinvariance}.) Then the function is never zero, and thus attains nonzero infimum value on every Grassmannian $\operatorname{Gr}(d,\mathcal{D}_R)$.
\end{proof}

\subsubsection{Geometric intersections and solution to kernel equation}
We will use the following abstract tube covering lemma for the kernel equation satisfied by a unit normal vector $A$.
\begin{lemma}\label{theorem4.4}
In a complex Hilbert space $\mathcal{H}$, let $E,Y\subset \mathcal{H}$ be linear subspaces. Denote by $P_{Y^\perp}$ the orthogonal projection onto $Y^\perp$ and define 
$$
T:=P_{Y^\perp}\mid_E:E\to Y^\perp.
$$For any threshold $\tau>0$, denote by 
$$
C_\tau(E,Y)\subset E
$$ the vector space spanned by all right singular vectors of $T$ with singular value at most $\tau$. Denote their dimensions by 
$$
c_\tau:=\dim_\mathbb{C}C_\tau(E,Y),\quad e:=\dim_\mathbb{C}E.
$$We then define the tuple 
$$
\mathcal{T}_\rho(E,Y):=\{x\in E:\|x\|\leq 2,\operatorname{dist}(x,Y)\leq\rho\}.
$$Then we have the following inclusion 
\begin{equation}
\label{followinginclusion}\mathcal{T}_\rho(E,Y)\subset (B_{C_\tau(E,Y)}(2))+\frac{\rho}{\tau}B_E(1).   
\end{equation}
More specifically, if we have the following orthogonal decomposition 
$$
E=C_\tau(E,Y)\oplus (C_\tau(E,Y))^\perp, 
$$then each element $x\in\mathcal{T}_p(E,Y)$ can be decomposed as 
\begin{equation}\label{line1196}
x=x_0+x_1,\quad x_0\in C_\tau(E,Y),\quad x_1\in (C_\tau(E,Y))^\perp\cap E,
\end{equation}satisfying that $\|x_0\|\leq 2,\|x_1\|\leq\rho/\tau$.

Therefore, for every mesh size $0<\epsilon<1$, $\mathcal{T}_\rho(E,Y)$ has an $\epsilon$-net of cardinality at most 
\begin{equation}\label{coveringbounds}
\mathcal{N}_\epsilon(\mathcal{T}_\rho(E,Y))\leq \left(\frac{C}{\epsilon}\right)^{2c_\tau}\left(\frac{C\max\{\rho/\tau,\epsilon\}}{\epsilon}\right)^{2(e-c_\tau)}.
\end{equation}
Moreover, $\mathcal{T}_\rho(E,Y)$ has an $\epsilon$-net $\mathcal{N}_\epsilon'$ contained entirely inside $T_\rho(E,Y)$ of cardinality
\begin{equation}
\mathcal{N}_\epsilon'(\mathcal{T}_\rho(E,Y))\leq \left(\frac{2C}{\epsilon}\right)^{2c_\tau}\left(\frac{2C\max\{\rho/\tau,\epsilon/2\}}{\epsilon}\right)^{2(e-c_\tau)}.
\end{equation}

\end{lemma}

\begin{proof}
 Take $x=x_0+x_1$ where $x_0\in C_\tau(E,Y),x_1\in (C_\tau(E,Y))^\perp$. Note that $(C_\tau(E,Y))^\perp$ is spanned by right singular vectors of $T$ with singular values larger than $\tau$. By the SVD, $Tx_0$ and $Tx_1$ live in orthogonal subspaces, so we must have 
 $$
\|Tx\|^2=\|Tx_0\|^2+\|Tx_1\|^2.
 $$Since $\operatorname{dist}(x,Y)\leq \rho$, we have $\|Tx_1\|\leq\rho$. But since on $(C_\tau(E,Y))^\perp$ singular values of $T$ are larger than $\tau$, we have $\|Tx_1\|\geq\tau\|x_1\|$. This forces $\|x_1\|\leq\rho/\tau$. Finally, $\|x_0\|\leq\|x\|\leq 2$ as $x_0$ is the orthogonal projection of $x$.
The covering bound \eqref{coveringbounds} is immediate since $\dim_\mathbb{C}(C_\tau(E,Y))^\perp=e-c_\tau$, and we take an $\epsilon/2$-net for the unit ball in $C_\tau(E,Y)$ as well as another $\epsilon/2$-net for the ball of radius $\rho/\tau$ in $(C_\tau(E,Y))^\perp$. Finally, by a standard procedure, an $\epsilon/2$-net $\mathcal{N}_\epsilon$ for $T_\rho(E,Y)$ yields an $\epsilon$-net $\mathcal{N}_\epsilon'$ for $T_\rho(E,Y)$ completely within $T_\rho(E,Y)$.
\end{proof}

Observe that if we take $Y=Y_\sigma(E)$ for some $\sigma>0$ as defined in \eqref{scaleysigmae}, then $C_\tau(E,Y)$ coincides with the space $C_{\tau,\sigma}(E)$ defined in \eqref{scaletausigmae}.

We explain how Lemma \ref{theorem4.4} is used for the kernel equation of normal vectors $A$.
If $\|(\mathcal{D}A)_k\|\leq\rho$, then we have 
$$
A_k\in E\cap (Y(E)+\rho B_\mathcal{H})
$$for each $k$. Then each $A_k$ is contained in $T_\rho(E,Y(E))$, so we apply Lemma \ref{theorem4.4} with $Y=Y(E)$. Then by \eqref{followinginclusion}, the entries $A_k$ are covered by the product of the net of the unit ball in $C_\tau(E,Y)$ and by the other net in its orthogonal complement with a much smaller net size. In practice, we always take $\epsilon=C_{op}\rho$ for a constant $C_{op}$ depending only on the operator norm, and $\epsilon,\rho$ will be polynomially small in $N$. The constant $\tau$ will be set sufficiently small (so we can use Proposition \ref{proposition1044}) but fixed and independent of $N$. Therefore, the net size on $E\cap(C_\tau(E,Y))^\perp $ used for approximating \eqref{line1196} is fixed no matter how small $\epsilon,\rho$ really are, and does not contribute in magnitude to the net entropy.

\subsubsection{Covariance operator and active subspace extraction}\label{section4.1.4}

We describe a standard procedure to project a PSD matrix onto subspaces where the least singular value is positively bounded from below. Let $E$ be a finite-dimensional Hilbert space with complex dimension $n$, and let 
$$
\mathcal{S}:=\sum_{i=1}^se_ie_i^*.
$$Then $\mathcal{S}\geq 0$. Let its eigenvalues be $\lambda_1,\cdots,\lambda_n\geq 0$, having the corresponding orthonormal eigenvectors $u_1,\cdots,u_n$.

Fix a threshold $\tau>0$. At the first layer, we test if $\lambda_n(\mathcal{S})\geq \tau$. If yes, we take $W=E$ and $e_i'=e_i$. If not, we remove the eigenvector $u_n$ and pass to $W_{n-1}:=\operatorname{span}\{u_1,\cdots,u_{n-1}\}$ and denote by $e_i^{(n-1)}:=P_{W_{n-1}}e_i$
and denote by 
\begin{equation}\label{stratification}
\mathcal{S}_{n-1}:=\sum_{i=1}^se_i^{(n-1)}(e_i^{(n-1)})^*=P_{W_{n-1}}\mathcal{S}P_{W_{n-1}}.
\end{equation}
Moreover, we have the approximation error
\begin{equation}\label{approximationerror}
\sum_{i=1}^s|e_i-e_i'|^2=\sum_{i=1}^s|(I-P_{W_{n-1}})e_i|^2=\operatorname{tr}((I-P_{W_{n-1}})\mathcal{S}(I-P_{W_{n-1}}))\leq \tau.
\end{equation}
In our applications, we will take $T_1,\cdots,T_s$ to be elements of $\operatorname{End}(E)$ and let $\Gamma_\mathcal{S}$ be the covariance operator of the Gaussian sum $\sum_{i=1}^sg_iT_i$: $\Gamma_\mathcal{S}=\operatorname{Cov}(\sum_{i=1}^sg_iT_i)$. In this case, the covariance is given by the following operator 
$$
\Gamma_\mathcal{S}:M_d(\mathbb{C})\to M_d(\mathbb{C}),
$$where 
\begin{equation}\label{equations8770}
\Gamma_\mathcal{S}:=\sum_{i=1}^s T_i\otimes T_i^*.
\end{equation}The same subspace stratification \eqref{stratification} as above can be applied to the PSD matrix $\Gamma_\mathcal{S}$. 

Then we consider a more general Gaussian with product matrices 
$$
X=\sum_{i=1}^m\sum_{j=1}^n g_{ij}T_iB_j
$$where $g_{ij}$ are i.i.d. standard complex Gaussians. Let $\mathcal{H}$ be a linear subspace of $M_d\mathbb(C)$ invariant under left $T_i$ action. For each $T_i$, the left multiplication operator is defined as 
$$
L_i:\mathcal{H}\to\mathcal{H},\quad L_i(B)=T_iB.
$$Then its Hilbert-Schmidt adjoint is $L_i^*(A)=T_i^*A$. 

Denote by the covariance operator of $B_j$ as 
\begin{equation}\label{equations877}
\Gamma_B:=\sum_{j=1}^n B_j\otimes B_j^*,
\end{equation}then one can verify that the covariance operator of $X$ is 
$$
\Gamma_X:=\sum_{i=1}^m L_i\Gamma_B L_i^*.
$$Furthermore, suppose that $B$ gives a linear subspace $V$ such that $\Gamma_B\geq\mu P_V$, then denote by 
$$
S_{T,V}:=\sum_{i=1}^m L_iP_VL_i^*,
$$one must have
\begin{equation}\label{composedvariance}
\Gamma_X\geq\mu S_{T,V}.
\end{equation}

\subsection{LSV via the net argument: active profile case}

When we apply a net argument to the kernel equation \begin{equation}
    (\mathfrak{D}_aB)_k=B_k+\sum_{i,\ell}\overline{X_\ell(i,k)}R_\ell^*B_i
\end{equation}we need to take into account two degeneracies:

\begin{enumerate}
    \item The covariance operator defined by $B_i$ is positively definite on the linear subspace spanned by $B_i$, but some eigenvalues are too small so that the operator is almost degenerate in some directions;
    \item The operator $\Psi_E$ on the subspace spanned by $E$ has enough nonzero singular directions, but some singular values are too small and some directions are nearly degenerate.
\end{enumerate}
We will first handle the second degeneracy by considering $Y_\sigma(E)$ for certain nonzero threshold values of $\sigma$ in this section. We first assume that the first degeneracy does not occur up to a threshold, see the following definition. The degenerate $B_i$-profile case will be settled in the next section via approximation.

\begin{Definition}
Fix positive integers $0<D_1<D_2<\cdots<D_{\dim \mathcal{D}_R}$. We say that a profile 
$$B=(B_1,\cdots,B_N)\in (\mathcal{D}_R)^N,\quad  \sum_i\|B_i\|_{HS}^2\in[\frac{1}{2},\frac{3}{2}],$$
lies in regularity class $d$, denoted $B\in\mathfrak{Reg}(d)$, if we can find a $d$-dimensional subspace $W\subset \mathcal{D}_R=\Sigma_R^\perp$, such that $W=\operatorname{span}(B_1,\cdots,B_N)$  and the covariance operator 
$$
\Gamma_B:=\sum_{j=1}^N B_j\otimes B_j^*
$$has smallest singular value at least $N^{-D_d}$ on $W$.

\end{Definition}

We shall use the following notation: for any measurable event $L\subset\Omega$,
$$
\mathbb{P}^{\Omega_{\text{good}}}(L)=\mathbb{P}(L\cap\Omega_{\text{good}}).
$$

Now we carry out the net argument for LSV, assuming $Y(E)=Y_\sigma(E)$ for some $\sigma>0$:

\begin{Proposition}\label{proposition1124}
Fix some $d\in\mathbb{N}_+$ and some $\sigma\in(0,\sigma_0)$, and consider the following set of subspaces 
$$
\mathcal{W}(d,\sigma):=\{E\in\operatorname{Gr}(d,\mathcal{D}_R):Y(E)=Y_\sigma(E)\}
.$$ Let $m_Z\in\mathbb{N}$ be such that $$\dim\mathcal{D}_R\leq m_Z^2.$$
Then we can find a constant $C_{\ref{proposition1124}}>1$ depending only on the module $Z$ and the coefficients $R_1,\cdots,R_n$ such that for any $\delta>0$,
   \begin{equation}\label{LHSOF1130}\begin{aligned}
&\mathbb{P}^{\Omega_{\text{good}}}(\inf_{B\in\mathfrak{Reg}(d), \operatorname{Span}\{B_i\}\in\mathcal{W}(d,\sigma)}\|\mathfrak{D}_aB\|\leq\delta)\\&\leq \delta^{2(N-1-m_Z^6)}(C_{\ref{proposition1124}}N)^{(D_d+2)2m_Z^2N}\sigma^{-4(N-1)m_Z^2}.\end{aligned}
   \end{equation}In particular, taking $
   \delta=\sigma^{4m_Z^2}(C_{\ref{proposition1124}}N)^{-8m_Z^2D_d-1}$, we get, whenever $N$ is sufficiently large, 
   $$
 \mathbb{P}^{\Omega_{\text{good}}}(\inf_{B\in\mathfrak{Reg}(d), \operatorname{Span}\{B_i\}\in\mathcal{W}(d,\sigma)}\|\mathfrak{D}_aB\|\leq\ \sigma^{4m_Z^2}(C_{\ref{proposition1124}}N)^{-8m_Z^2D_d-1})\leq\exp(-N\log N).$$
\end{Proposition}
The restriction to $\mathcal{W}(d,\sigma)$ means that we assume all nonzero singular values of $\Psi_E$ are larger than $\sigma$. We will remove this restriction later. 
\begin{proof}
   Assume that $(B_i)_{i=1}^N$ spans the subspace $E$ and the covariance lower bound $\Gamma_{B}\geq N^{-D_d}I_E$ holds for $B$ on $E$. Then take the $k$-th coordinate for $\|\mathfrak{D}_aB\|_2\leq\delta$, we have
    $$
\|B_k+\sum_{i,\ell}\overline{X_\ell(i,k)}R_\ell^*B_i\|_{op}\leq  \delta,\quad\forall k:k\neq a.
    $$
    This inclusion implies that, for each $1\leq k\leq N:k\neq a$, 
    $$
B_k\in B_E(2)\cap( Y(E)+\delta B_{M_m(\mathbb{C})})=\mathcal{T}_{\delta}(E,Y(E)).
    $$Since we assume $Y(E)=Y_\sigma(E)$ for these $E$, we apply Lemma \ref{theorem4.4} 
to get that we can find an $\epsilon=\frac{\delta}{16M_{\text{op}}\sqrt{N}}$-net $\mathcal{N}_\epsilon$ for $\mathcal{T}_{\delta}(E,Y(E))$ (and that $\mathcal{N}_\epsilon\subset\mathcal{T}_\delta(E,Y(E))$ of cardinality
\begin{equation}\label{netarguments}|\mathcal{N}_\epsilon|\leq (\frac{C}{\epsilon})^{2\dim C_{\tau,\sigma}(E)}(\frac{16CM_{\text{op}}\sqrt{N}}{\tau})^{2m_Z^2}
\end{equation}
which can serve as a net for each $B_k,k\neq a$ (here we take $\tau>0$ sufficiently small, say $\tau<\tau_0$, but fixed and independent of $N$). For the remaining coordinate $B_a$, we simply take an $\epsilon$-net $\mathcal{N}_\epsilon^0\subset E$ of cardinality $|\mathcal{N}_\epsilon^0|\leq (\frac{C}{\epsilon})^{2m_Z^2}$. For each $k\neq a$, take $B_k'\in\mathcal{N}_\epsilon$ with $|B_k-B_k'|\leq\epsilon$ and take $B_a'\in\mathcal{N}_\epsilon^0,\|B_a-B_a'\|\leq\epsilon$. In this proof we shall assume $\delta\leq N^{-D_d}$, otherwise the estimate is trivial. Then the covariance operator $\Gamma_{B'}$ defined by $B'$ still satisfies that 
$$
\Gamma_{B'}\geq \frac{1}{2}N^{-D_d}I_E.
$$
By the definition of $\epsilon$, on the good operator norm event $\Omega_{\text{good}}$ we have
    \begin{equation}\label{approxkernelequation}
        \|B_k'+\sum_{i,\ell}\overline{X_\ell(i,k)}R_\ell^*B_i'\|_{op}\leq  2\delta,\quad\forall k:k\neq a.
    \end{equation}
    The covariance operator of the following Gaussian 
    \begin{equation}\label{gaussiangk}
\mathcal{G}_{k}:=\sum_{i,\ell}\overline{X_\ell(i,k)}R_\ell^*B_i'
 \end{equation} has covariance bounded from below by $\frac{1}{N}\cdot\frac{1}{2} N^{-D_d}\cdot \sigma^2$ on a linear subspace of dimension $\dim Y_\sigma(E)$: this follows from combining (1) the computation in \eqref{composedvariance}; (2) that each $X_\ell(i,k)$ has variance $\frac{1}{N}$; (3) the construction that $\Gamma_{B'}\geq\frac{1}{2} N^{-D_d}$ on $I_E$; and (4) the definition that $Y_\sigma(E)$ generates the linear subspace supported by left singular vectors of $\Psi_E$ of singular value larger than $\sigma$. We write this more formally as 
    \begin{equation}\label{covgk}
\operatorname{Cov}(\mathcal{G}_k)\geq \frac{\sigma^2}{2}N^{-D_d-1} I_{Y_\sigma(E)}.
    \end{equation}

    Therefore, for each $k\neq a$, by standard Gaussian anti-concentration,
    \begin{equation}\label{lines1186}
\mathbb{P}(\|B_k'+\sum_{i,\ell}\overline{X_\ell(i,k)}R_\ell^*B_i'\|_{op}\leq 2\delta)\leq (\frac{2C'\delta^2}{\sigma^2N^{-D_d-1}})^{\dim Y_\sigma(E)}
\end{equation}for a universal constant $C'>0$ depending only on the dimension upper bound $m_Z$.

    Taking a union bound over the net, we see that on the operator norm good event, 
    $$\begin{aligned}&
\mathbb{P}^{\Omega_{\text{good}}}(\inf_{B\in\mathfrak{Reg}(d): \operatorname{Span}\{B_i\}=E}\|\mathfrak{D}_aB\|\leq\delta)\\&\leq (\frac{16CM_{\text{op}}\sqrt{N}}{\delta})^{2(N-1)\dim C_{\tau,\sigma}(E)+2m_Z^2}(\frac{16CM_{\text{op}}\sqrt{N}}{\tau})^{2m_Z^2N}
(\frac{2C'\delta^2}{\sigma^2N^{-D_d-1}})^{(N-1)\dim Y_\sigma(E)}.\end{aligned}$$
Since $\dim Y_\sigma(E)\geq\dim C_{\tau,\sigma}(E)+1$ by Proposition \ref{proposition1044}, we simplify as
\begin{equation}\label{wesimplifyas}
\mathbb{P}^{\Omega_{\text{good}}}(\inf_{B\in\mathfrak{Reg}(d): \operatorname{Span}\{B_i\}=E}\|\mathfrak{D}_aB\|\leq\delta)\leq 
\delta^{2(N-1)-2m_Z^2}(C'')^{N}N^{(D_d+2)2m_Z^2N}\sigma^{-4(N-1)m_Z^2}\end{equation}where $C''>0$ depends only on $\tau,C'$ and $M_{\text{op}}$. This proves the claimed estimate for a fixed subspace $E$. 

To consider all subspaces $E\in\mathcal{W}(d,\sigma)$ of dimension $d$, we take a $\frac{\delta}{16M_{\text{op}}\sqrt{N}}$-net of  $\mathcal{W}(d,\sigma)\subset\operatorname{Gr}(d,\mathcal{D}_R)$, which is contained in $\mathcal{W}(d,\sigma)$, of cardinality $(1+\frac{CM_{\text{op}}\sqrt{N}}{\delta})^{2m_Z^2d(m_Z^2-d)}$. In order for \eqref{LHSOF1130} to yield a nontrivial estimate, we consider any $\delta<N^{-D_d-1}$ in the following. Then for $B\in\mathfrak{Reg}(d)$ spanning $E\in\mathcal{W}(d,\sigma)$, we find $E'\in \mathcal{W}(d,\sigma)$ with $d_{\text{Gr}}(E,E')\leq \frac{\delta}{16M_{\text{op}}\sqrt{N}}$. We define $B_i'=P_{E'}B_i$, with $P_{E'}$ the orthogonal projection onto $E'$, and then one verifies that $\sum_i\|B_i-B_i'\|_{HS}^2\leq\frac{\delta^2}{64M_{\text{op}}^2N}$\footnote{By definition of Grassmannian distance, $d_{\operatorname{Gr}}(E,E')=\|P_E-P_{E'}\|_{\text{op}}$.}, that the covariance operator of $B'$ satisfies $\Gamma_{B'}\geq\frac{1}{2}N^{-D_d}$ on $I_{E'}$, and that the approximate kernel equation \eqref{approxkernelequation} holds for this $B'$.
Then we use the estimate \eqref{wesimplifyas} for this $E'$ with $\delta$ replaced by $2\delta$, and note $2m_Z^2d(m_Z^2-d)\leq 2m_Z^6$. This completes the proof after summing the bad operator norm event. The exponent $m_Z^6$ is an upper bound absorbing the size of a net for $B_a$ and for the Grassmannian $\operatorname{Gr}(d,\mathcal{D}_R)$.
\end{proof}

We now remove the restriction to $\mathcal{W}(d,\sigma)$ in Proposition \ref{proposition1124}. The idea is as follows:

\begin{fact}\label{fact4.80}
    Fix a dimension $1\leq d\leq \dim \mathcal{D}_R$ and a given integer $D_d>0$. Denote by $\sigma_1^*(d):=\sigma_0$ for the $\sigma_0$ defined in Proposition \ref{proposition1044}. For any $d$-dimensional subspace $E\subset \mathcal{D}_R$, denote by $\sigma_1\geq\sigma_2\cdots\geq\sigma_{dn}$ its largest singular values of $\Psi_E$ in decreasing order (this exhausts all positive singular values of $\Psi_E$ and the sequence $\sigma_i$ is padded by zeros up to length $dn$). Then $\sigma_1\geq\sigma_1^*(d)$ by Proposition \ref{proposition1044}. For each $2\leq k\leq dn$, we iteratively define    \begin{equation}\label{iterativedefinition}
\sigma_k^*(d)=(\sigma_{k-1}^*(d))^{4m_Z^2}(C_{\ref{proposition1124}}N)^{-8m_Z^2D_d-3}.\end{equation}Then we can find a constant $\operatorname{Mag}_d(D_d)>0$ depending only on $m_Z,n$ and $D_d$  such that 
    $$
\sigma_{dn}^*(d)\geq N^{-\operatorname{Mag}_d(D_d)}
\text{   \quad whenever $N$ is larger than some fixed constant.}    $$ (The lower bound in $N$ depends on $C_{\ref{proposition1124}}$ but $\operatorname{Mag}_d(D_d)$ does not.) We have the following dichotomy for the singular value chains $\sigma_1,\cdots,\sigma_{dn}$:
\begin{enumerate}
    \item Either $\sigma_{dn}\geq \sigma_{dn}^*(d)$,  
    \item Or there is a large gap among those singular values smaller than $\sigma_0$:\begin{enumerate} \item Either we can find some $k$, $2\leq k\leq dn$ such that $\sigma_0\geq \sigma_{k-1}\geq\sigma_{dn}^*(d)$ but 
    \begin{equation}\label{theassumptions}
\sigma_k\leq (\sigma_{k-1})^{4m_Z^2}(C_{\ref{proposition1124}}N)^{-8m_Z^2D_d-3},
    \end{equation}
    \item Or we can find some $2\leq k\leq dn$ such that $\sigma_{k-1}>\sigma_0$ but
    $$
\sigma_k\leq (\sigma_{0})^{4m_Z^2}(C_{\ref{proposition1124}}N)^{-8m_Z^2D_d-3}.
    $$
    \end{enumerate}
\end{enumerate}
\end{fact}
The proof of this fact is evident. Since the iteration \eqref{iterativedefinition} is only done finitely many times, we must be able to find the finite constant $\operatorname{Mag_d(D_d)}$, which is essentially a function of $D_d$ (save for the fixed parameter $m,d$.) We take a dyadic decomposition for the range:

\begin{fact}\label{fact1203}
Suppose that the subspace $E$ satisfies alternative (2)(a) of Fact \ref{fact4.80}. Then we take a dyadic decomposition $\mathfrak{Dya}$ for the range $[\sigma_{dn}^*(d),\sigma_0]$, which contains $\sigma_0$, such that for any $\sigma_{k-1}$ in this interval, we can find a value $x\in\mathfrak{Dya}$ such that $x\leq \sigma_{k-1}\leq 2x$. We have $|\mathfrak{Dya}|\leq \operatorname{Mag}_d(D_d)\log N+O(1)$. The assumption \eqref{theassumptions} implies that for any $dn\geq k'\geq k$, $\sigma_{k'}\leq (2x)^{4m_Z^2}(C_{\ref{proposition1124}}N)^{-8m_Z^2D_d-3}.$     Otherwise, if $E$ satisfies alternative (2)(b), then we take $x=\sigma_0$ when applying Proposition \ref{proposition1124news}.
\end{fact}

When there is a large gap in the singular values, we arrive at the following similar conclusion as in Proposition \ref{proposition1124}:

\begin{Proposition}\label{proposition1124news}
Fix some $d\in\mathbb{N}_+$, some $x\in\mathfrak{Dya}$, and consider the following set of subspaces 
$$\begin{aligned}
\mathcal{W}^{\text{gap}}(d,x):=\{E\in&\operatorname{Gr}(d,\mathcal{D}_R):\exists 2\leq k\leq dn\text{ such that }\\& \sigma_{k-1}\geq x,\sigma_k\leq (2x)^{4m_Z^2}(C_{\ref{proposition1124}}N)^{-8m_Z^2D_d-3}\},
\end{aligned}$$ where $\sigma_1\geq\sigma_2\geq\cdots\geq\sigma_{dn}$ are singular values of $\Psi_E$. (Note that when $N$ is sufficiently large we have $(2x)^{4m_Z^2}(C_{\ref{proposition1124}}N)^{-8m_Z^2D_d-3}<x$.)
Then we can find a constant $C_{\ref{proposition1124news}}>1$ depending only on the coefficients $R_1,\cdots,R_n$ such that  whenever $N$ is sufficiently large, 
   $$
 \mathbb{P}^{\Omega_{\text{good}}}(\inf_{B\in\mathfrak{Reg}(d), \operatorname{Span}\{B_i\}\in\mathcal{W}^{\text{gap}}(d,x)}\|\mathfrak{D}_aB\|\leq\ x^{4m_Z^2}(C_{\ref{proposition1124news}}N)^{-8m_Z^2D_d-1})\leq\exp(-N\log N).$$
\end{Proposition}

\begin{proof}
   Recall that $\Psi_E^{\geq x}$ is the operator $P_{Y_x(E)}\Psi_E$, that is, only preserving singular value directions of $\Psi_E$ with singular value $\geq x$, and we denote by $\Psi_E^{< x}=\Psi_E-\Psi_E^{\geq x}$. Then we define, for each subspace $E\in\operatorname{Gr}(d,\mathcal{D}_R)$, the (random) operator 
    $$
\mathfrak{D}_{a,E}^{\geq x}:E^N\to\operatorname{End}(Z)^{N-1}
$$ 
 
\begin{equation}
    (\mathfrak{D}_{a,E}^{\geq x}B)_k=B_k+\Psi_E^{\geq x}(\sum_i \overline{X_1(i,k)}B_i,\cdots,\sum_i\overline{X_n(i,k)}B_i),\quad k\neq a.
\end{equation}
Similarly we define 
 $$
\mathfrak{D}_{a,E}^{<x}:E^N\to\operatorname{End}(Z)^{N-1}
$$ 
 
\begin{equation}
    (\mathfrak{D}_{a,E}^{<x}B)_k=\Psi_E^{<x}(\sum_i \overline{X_1(i,k)}B_i,\cdots,\sum_i\overline{X_n(i,k)}B_i),\quad k\neq a
\end{equation}

Then we can verify that for any $B=(B_1,\cdots,B_N)\in E^N$, $$\mathfrak{D}_aB=\mathfrak{D}_{a,E}^{\geq x}B+\mathfrak{D}_{a,E}^{<x}B,$$ and that for any $B$ with $\sum_i\|B_i\|_{HS}^2\leq 2$, on the event $\Omega_{\text{good}}$, we have 
$$
|(\mathfrak{D}_{a,E}^{<x}B)_k|\leq 16\sqrt{n}\|\Psi_E^{<x}\|\quad\forall k\neq a\Rightarrow\|\mathfrak{D}_{a,E}^{<x}B\|\leq 16\sqrt{Nn}\|\Psi_E^{<x}\|.
$$ The assumption on $E\in\mathcal{W}^{\text{gap}}(d,x)$ implies that, for large enough $N$, 
\begin{equation}\label{boundone}
\|\mathfrak{D}_{a,E}^{<x}B\|\leq 16\sqrt{Nn} \cdot (2x)^{4m_Z^2}(C_{\ref{proposition1124}}N)^{-8m_Z^2D_d-3}.
\end{equation}
Denote by $y$ the quantity on the right-hand side 
$$
y:=16\sqrt{Nn} \cdot (2x)^{4m_Z^2}(C_{\ref{proposition1124}}N)^{-8m_Z^2D_d-3}
.$$
 Assume that $(B_i)_{i=1}^N$ spans the subspace $E$, then we project $\|\mathfrak{D}_aB\|_2\leq\delta$ to each coordinate:
    $$
\|B_k+\sum_{i,\ell}\overline{X_\ell(i,k)}R_\ell^*B_i\|_{op}\leq  \delta,\quad\forall k:k\neq a.
    $$ Since the random part of the image of $\mathfrak{D}_{a,E}^{\geq x}$ (the part not including the term $B_k$) lies in $Y_x(E)$ and $\|\mathfrak{D}_{a,E}^{<x}B\|\leq y$ by \eqref{boundone}, we have for each $1\leq k\leq N:k\neq a$, 
    $$
B_k\in B_E(2)\cap( Y_x(E)+(\delta+y) B_{M_m(\mathbb{C})})=\mathcal{T}_{\delta+y}(E,Y_x(E)).
    $$Throughout the rest of the proof we assume $\delta\geq y$, so that $\mathcal{T}_{\delta+y}(E,Y_x(E))\subset \mathcal{T}_{2\delta}(E,Y_x(E))$.
Then we are in exactly the same situation as in Proposition \ref{proposition1124} with $\delta$ replaced by $2\delta$, since $\dim C_{\tau}(E,Y_x(E))+1\leq\dim Y_x(E)$ by Proposition
\ref{proposition1044} (note that $x\leq \sigma_0$.)
Then the proof of Proposition \ref{proposition1124} applies verbatim, first for one fixed $E$ and then after taking the same Grassmannian discretization of $\mathcal{W}^{\text{gap}}(d,x)$. Therefore, the following estimate holds for any $N^{-D_d-1}\geq \delta\geq y$:  
   \begin{equation}\label{LHSOF1130sigmas}\begin{aligned}
&\mathbb{P}^{\Omega_{\text{good}}}(\inf_{B\in\mathfrak{Reg}(d), \operatorname{Span}\{B_i\}\in\mathcal{W}^{\text{gap}}(d,x)}\|\mathfrak{D}_aB\|\leq\delta)\\&\leq (2\delta)^{2(N-1-m_Z^6)}(C_{\ref{proposition1124}}N)^{(D_d+2)2m_Z^2N}x^{-4(N-1)m_Z^2}.\end{aligned}
   \end{equation}
   Finally we take $
   \delta=x^{4m_Z^2}(C_{\ref{proposition1124}}N)^{-8m_Z^2D_d-1}$, and we use a constant $C_{\ref{proposition1124news}}>C_{\ref{proposition1124}}$ to absorb all the fixed constants and finish with the estimate involving $C_{\ref{proposition1124news}}$. Clearly, when $N$ is sufficiently large, then $\delta\geq y$ so that this choice is valid. This completes the proof.
\end{proof}

Combining all the above cases, we conclude with the following Proposition:

\begin{Proposition}\label{proposition1124final}
Fix some $d\in\mathbb{N}_+$ and some $D_d\in\mathbb{N}_+$.
Then we can find an integer $\operatorname{Deg}_d(D_d)>0$ depending only on $m_Z,d,D_d$ such that, whenever $N$ is sufficiently large (say $N\geq N_0=N_0(Z,R_1,\cdots,R_n,m_Z,d,D_d)$), we have
   $$
 \mathbb{P}^{\Omega_{\text{good}}}(\inf_{B\in\mathfrak{Reg}(d)}\|\mathfrak{D}_aB\|_2\leq N^{-\operatorname{Deg}_d(D_d)} )\leq\exp(-N\log N).$$
\end{Proposition}

\begin{proof}
For any subspace $W$ of dimension $d$, either alternative (1) of Fact \ref{fact4.80} is fulfilled and we apply Proposition \ref{proposition1124} with $\sigma=\sigma_{dn}^*(d)$. Or alternative (2) of Fact \ref{fact4.80} is satisfied, and we apply Proposition \ref{proposition1124news} for each $x\in\mathfrak{Dya}$ thanks to the classification in Fact \ref{fact1203}. Taking a union bound over all these events and over the dyadic decomposition $\mathfrak{Dya}$ completes the proof.     
\end{proof}

\subsection{Non-regular profiles and proof completion}

We shall use the following simple approximation result for a profile $B$ by a regular profile $B'$ living in a vector space of smaller dimension:
\begin{lemma}\label{approximationlemma4.6}
    For any profile $B=(B_1,\cdots,B_N)\in (\mathcal{D}_R)^N$ with $\sum_i\|B_i\|_{HS}^2=1$, if $B\notin\mathfrak{Reg}(\dim \mathcal{D}_R)$, then we can find some integer $d',1\leq d'\leq\dim \mathcal{D}_R-1$ and an element $(B_1',\cdots,B_N')\in \mathfrak{Reg}(d')$, with an approximation 
    \begin{equation}\label{approxerror}
\sum_{i=1}^N\|B_i-B_i'\|_{HS}^2\leq N^{-D_{d'+1}}+N^{-D_{d'+2}}+\cdots+N^{-D_{\dim\mathcal{D}_R}}
    \end{equation} whenever $N$ is sufficiently large.
\end{lemma}
\begin{proof}
For the covariance operator $\Gamma_B$ on $\mathcal{D}_R$ defined in \eqref{equations877}, if its smallest singular value is less than $N^{-D_{\dim\mathcal{D}_R}}$, then we apply the procedure described in Section \ref{section4.1.4} to find linear subspaces $W$ of $\mathcal{D}_R$ of codimension 1 and project each $B_i$ onto $W$ and get $B_i(1)$. Then we define the covariance operator $\Gamma_{B(1)}$ on $W$. The approximation error is as in \eqref{approximationerror}. Then we test its smallest singular value on $W$: if it is still smaller than $N^{-D_{\dim W}}$, we do this again and project to an even smaller subspace. The process must terminate since the largest eigenvalue of $\Gamma_B$ is at least $1/m_Z^2$ (since the trace is 1 and the dimension is at most $m_Z^2$) and likewise whenever $N$ is large enough, the largest eigenvalue of each covariance operator $\Gamma_{B(k)}$ of approximate vector $B(k)$ is at least $1/2m_Z^2$; therefore if all previous steps of approximation do not yield termination, we must terminate at dimension 1. For $N$ large, the approximation error (right-hand side of \eqref{approxerror}) is less than $\frac{1}{2}$, so the approximation vector belongs to $\mathfrak{Reg}(d')$ for this value of $d'$.
\end{proof}

Now we can complete the proof of Theorem \ref{theorem3.4}. The idea is to use Lemma \ref{approximationlemma4.6} and choose the constants $D_1,D_2,\cdots$ inductively such that $D_i$ is much larger than $D_{i-1}$ and thus the approximation error does not ruin the established least singular value estimate.

\begin{proof}[\proofname\ of Theorem \ref{theorem3.4}]
     We first take $D_1=1$. Then we take an integer $D_2$ satisfying $D_2\geq 2\operatorname{Deg}_1(D_1)+3$, and then take a $D_3$ satisfying $D_3\geq 2\operatorname{Deg}_2(D_2)+3$, iteratively to define all the indices $D_i,1\leq i\leq\dim \mathcal{D}_R$. Here $\operatorname{Deg}_i(D_i)$ are constants determined by Proposition \ref{proposition1124final}. We check that on the good operator event $\Omega_{\text{good}}$, these choices of $D_i$ yield the desired estimate. Indeed, if $B\in\mathfrak{Reg}(\dim\mathcal{D}_R)$, then we use Proposition \ref{proposition1124final}. Otherwise, by Lemma \ref{approximationlemma4.6}, if $B\notin\mathfrak{Reg}(\dim\mathcal{D}_R)$, then we can find some $d'\leq \dim \mathcal{D}_R-1$ and another element $B'=(B_1',\cdots,B_N')\in\mathfrak{Reg}(d')$ such that 
    $$
\sum_{i=1}^N \|B_i-B_i'\|_{HS}^2\leq m_Z^2N^{-D_{d'+1}}.
    $$For the profile $B'$, we have by Proposition \ref{proposition1124final} that the following lower bound$$
    \|\mathfrak{D}_aB'\|_2\geq N^{-\operatorname{Deg}_{d'}(D_{d'})}
    $$
    holds on $\Omega_{\text{good}}$ with probability at least $1-\exp(-N\log N)$. Then on $\Omega_{\text{good}}$ we have 
    $$
\|\mathfrak{D}_a(B-B')\|_2\leq 16m_ZM_{\text{op}}N^{-D_{d'+1}/2}.
    $$Since $D_{d'+1}\geq2\operatorname{Deg}_{d'}(D_{d'})+3$, we have, whenever $N$ is sufficiently large, that
$$
\|\mathfrak{D}_aB\|_2\geq\frac{1}{2} N^{-\operatorname{Deg}_{d'}(D_{d'})}
    $$
    This completes the proof as we combine estimates from all the cases $1\leq d'\leq\dim \mathcal{D}_R$ and slightly changing the exponent in $N$ to absorb the $\frac{1}{2}$ factor. Finally we use $\mathbb{P}(\Omega_{\text{good}}^c)\leq\exp(-\Omega(N)).$
\end{proof}

\section{Brown measure limit and convergence to Brown measure}\label{section5555}
In this section we complete the proof of Theorem \ref{mainconvergencetheorem}. This part of the argument is essentially the same as in \cite{cook2022spectrum}, Section 8. While \cite{cook2022spectrum} only considered quadratic polynomials in Ginibre because their least singular value estimates are only proven for quadratics, the resulting convergence to Brown measure conditioned on a known least singular value estimate holds for polynomials of any degree. We still present the main steps here not only for the sake of completeness, but also for two more reasons: (1) our least singular value estimate does not hold for $z=\mathfrak{p}(0)$ and is not uniform for $z$, but we show that this does not prevent us from using the same convergence argument; and (2) when we adapt the proof to non-Gaussian cases in Section \ref{section6}, several technical steps need to be reworked.

\label{section5.2for}

 For any $z\in\mathbb{C}$ we set 
$$
\nu_N^z:=\mu_{(zI_N-P^N)(zI_N-P^N)^*}
$$where we denote $P^N=\mathfrak{p}(X_1^N,\cdots,X_n^N)$. Applying Green's formula we have
\begin{equation}
\label{thiswillverify}\int_\mathbb{C}\psi(z)d\mu_{P^N}(z)=\frac{1}{4\pi}\int_\mathbb{C}\Delta\psi(z)\int_0^\infty\log(x)d\nu_N^z(x)dz.
\end{equation}Then to justify the convergence, we prove in the following three steps:
\begin{enumerate}
    \item\label{step11}  For each fixed $z\in\mathbb{C}$ the measure $\nu_N^z$ converges weakly almost surely to a probability measure $\nu^z$ on $\mathbb{R}_+$ which can be identified.
    \item\label{step22} For almost every $z\in\mathbb{C}$, we have that the integral $\int_{\mathbb{R}_+}\log (x)d\nu_N^z(x)$ converges to $\int_{\mathbb{R}_+}\log(x)d\nu^z(x)$ in probability. This step uses Theorem \ref{theoremonline89}.
    \item\label{step33}
    In probability, the function $z\mapsto \int_{\mathbb{R}_+}\log(x)d\nu_N^z(x)$ converges in $L^1$ to the function $z\mapsto \int_{\mathbb{R}_+}\log(x)d\nu^z(x)$ so that $\mu_{P^N}$ converges in distribution to the limit that is identified as $\nu_{\mathfrak{p}}(c_1,\cdots,c_n)$.
\end{enumerate}
For non-Gaussian distributions, the proof in Step \ref{step33} does not need any change, but the proof of Step \ref{step11} and Step \ref{step22} will be modified in Section \ref{section6}.

Proof of Step \ref{step11}. Since $(zI-P^N)(zI-P^N)^*$ is self-adjoint in independent Ginibre matrices and their adjoints, this part directly follows from Voiculescu’s theorem \cite{voiculescu1991limit}. The limiting measure $\nu^z$ is the law of $|z-\mathfrak{p}(c_1,\cdots,c_n)|^2$ in the sense of $*$-moments. Moreover, $\nu^z$ is compactly supported for any fixed $z\in\mathbb{C}$.

Proof of Step \ref{step22}.
Denote the event 
$$
\mathcal{G}_N(z):=\{\|P^N\|\leq C_\mathfrak{p},\quad\sigma_{\text{min}}(z-P^N)\geq N^{-\beta/2}\}.
$$Then by Theorem \ref{theoremonline89} and Fact \ref{fact2.7}, we can choose $\beta$ large enough so that for any $z\in\mathbb{C}:z\neq \mathfrak{p}(0)$, $\mathbb{P}(\mathcal{G}_N(z))$ tends to one (although the rate may not be uniform in $z$). 

For any smooth nonnegative function $\chi_\epsilon$ vanishing on $[0,\epsilon/2]$ and equaling one on $[\epsilon,\infty)$,  since we have that 
$$
\limsup_{N\to\infty}\|P^N\|\leq C_\mathfrak{p}\quad a.s.,
$$then Step \ref{step11} tells us that we have convergence 
\begin{equation}
\lim_{N\to\infty}\int _0^\infty \chi_\epsilon(x)\log(x)d\nu_N^z(x)=\int_0^\infty\chi_\epsilon(x)\log(x) d\nu^z(x)\quad 
    a.s.
\end{equation}
Therefore, we only need to prove that for any $\delta,\delta'>0$ we can find an $\epsilon\in(0,1)$ satisfying that 
$$
\limsup_{N\to\infty}\mathbb{P}\left\{\mathbf{1}_{\mathcal{G}_N(z)}\left|
\int_0^\epsilon\log x d\nu_N^z(x)
\right|\geq\delta\right\}
\leq\delta',$$ which further boils down to proving, for each $z\neq \mathfrak{p}(0)$,
$$
\lim_{\epsilon\to 0}\limsup_{N\to\infty}\mathbb{E}\left|\int_{N^{-\beta}}^\epsilon\log xd\nu_N^z(x)\right|=0,
$$
since there are no singular values in $[0,N^{-\beta}]$ on the event $\mathcal{G}_N(z)$.

We define the Stieltjes transform of $\nu^z$ and $\mathbb{E}[\nu_N^z]$ as follows: for any $\zeta\in\mathbb{C}$,
$$
g^z(\zeta)=\int_0^\infty\frac{1}{\zeta-x}d\nu^z(x),\quad g_N^z(\zeta)=\mathbb{E}[\int_0^\infty \frac{1}{\zeta-x}d\nu_N^z(x)]
.$$
We only consider a polynomial $\mathfrak{p}$ of degree at least 1. By \cite{cook2022spectrum}, Lemma 8.1 (the proof actually works for $\mathfrak{p}$ of all degrees, see Appendix \ref{AppendixA} for a proof recall), for any fixed $z\in\mathbb{C}$ we can find $C,N_0>0$ and $c_1,c_2\in(0,1)$ satisfying that for any $\eta\in[N^{-c_1},1]$ and $N\geq N_0$, we have
\begin{equation}\label{thisisfromref3}
|\Im (g_N^z(i\eta))|\leq C\eta^{-c_2}.
\end{equation}
This estimate implies that for any $x\in[N^{-c_1},1]$, we have
\begin{equation}\label{bythegaussianmodel}
F_N^z(x):=\mathbb{E}[\nu_N^z([0,x])]\leq 2x|\Im(g_N^z(ix))|\leq 2Cx^{1-c_2}.
\end{equation} Therefore, for some $\alpha<c_1$ and any $\epsilon<1$,
$$
\begin{aligned}
&\mathbb{E}|\int_{N^{-\beta}}^\epsilon\log xd\nu_N^z(x)|\\&\leq -\int_{N^{-\beta}}^{N^{-\alpha}}\log xd\mathbb{E}[\nu_N^z(x)]-\int_{N^{-\alpha}}^{\epsilon}\log xd\mathbb{E}[\nu_N^z(x)]\\&\leq 2C\beta(\log N)N^{-\alpha(1-c_2)}+2C\frac{\epsilon^{1-c_2}}{1-c_2}-2C(\log\epsilon)\epsilon^{1-c_2},
\end{aligned}$$ which tends to 0 as $\epsilon$ is small enough. This completes Step \ref{step22}.

Proof of Step \ref{step33}. Define the following two functions 
$$
h_N(z):=\int_0^\infty\log xd\nu_N^z(x),\quad h(z)=\int_0^\infty\log xd\nu^z(x).
$$We have proven $h_N(z)$ converges to $h(z)$ in probability for any $z\neq\mathfrak{p}(0)$. For any compact subset $K$, since the operator norm event $\mathcal{A}_N:=\{\|P^N\|\leq C_\mathfrak{p}\}$ has probability $1-o(1)$, we only need to verify the $L^1$ convergence
\begin{equation}
 \lim_{N\to\infty}\mathbb{E}\left(\mathbf{1}_{\mathcal{A}_N}\int_{z\in K}|h_N(z)-h(z)|dz
 \right)=0,   
\end{equation}
and this will verify \eqref{thiswillverify} for twice continuously differentiable functions of bounded support. The bounded support condition is again removed since eigenvalues are almost surely bounded. To verify the convergence we check that they are uniformly bounded in $L^2$: by definition, 
$$
h_N(z)=\frac{2}{N}\sum_{j=1}^N\log|z-\lambda_j(P^N)|,
$$
then by Jensen,
\begin{equation}\label{ell2integrability}
\mathbb{E}\left[\mathbf{1}_{\mathcal{A}_N}\int_K|h_N(z)|^2dz\right]\leq 4\sup_{|\lambda|\leq C_\mathfrak{p}}\int_K|\log|z-\lambda||^2dz,
\end{equation}and the same $L^2$ norm upper bound holds for $h$. Then $\mathbf{1}_{\mathcal{A}_N}h_N$ converges to $\mathbf{1}_{\mathcal{A}_N}h$ in $L^1(\Omega\times K)$. Thus $h_N$ converges in probability to $h$ in $L^1(K)$, so $\mu_{P^N}$ converges to $\nu_{\mathfrak{p}(c_1,\cdots,c_n)}$ in probability in the sense of distributions on $K$. Enlarging $K$ to contain all spectrum of $\mathfrak{p}(c_1,\cdots,c_n)$ completes the proof of Theorem \ref{mainconvergencetheorem}.

\section{Extension to non-Gaussian cases}
\label{section6}
This section proves the non-Gaussian entry convergence result: Theorem \ref{nongaussiantheorem}. We only need to show that each specific Gaussian computation in this paper can be generalized to the non-Gaussian case, possibly with a weakened exponent.

We begin with an operator norm control:

\begin{fact}
    Let $\xi$ be a complex random variable with mean 0, variance 1 and having finite moments of all orders. Let $X^N$ be an $N\times N$ matrix with i.i.d. entries of law $\frac{1}{\sqrt{N}}\xi$. Then for any $q\in\mathbb{N}$ we can find $C_{q,\xi}>0$ depending on $q,\xi$ such that  \begin{equation}\label{thereforeoperator}\mathbb{P}(\|X^N\|\leq 8)\geq 1-C_{q,\xi}N^{-q}.\end{equation} Therefore, in the setting of Theorem \ref{nongaussiantheorem} we work under the following high probability event 
    $$
\Omega_{\text{high-prob}}:=\{\|X_\ell\|\leq 8\quad \forall 1\leq\ell\leq n\},
    $$and since $n$ is fixed we have, for any $q\in\mathbb{N}$ there is a $C_{q,\xi,n}$ depending on $q,\xi,n$:
    $$
\mathbb{P}(\Omega_{\text{high-prob}})\geq 1-C_{q,\xi,n}N^{-q}.
    $$
\end{fact}
\begin{proof}
    The operator norm control \eqref{thereforeoperator} with non-asymptotic polynomial rate error can be found in  \cite{tao2013outliers}, Theorem 1.4.
\end{proof}

\subsection{Least singular value estimates}We first prove analogous least singular value estimates.
We can generalize the proof of Theorem \ref{theorem3.4}:

\begin{Proposition}In the setting of Theorem \ref{nongaussiantheorem}, the conclusion of Theorem \ref{theorem3.4} still holds in the following form: we can find $D_Z>0$ depending only on $\dim Z$, and for each $q\in\mathbb{N}_+$ we can find some  $C_{q,\xi}>0$ (depending only on $q,\xi,n,\dim Z$) and some $N_0=N_0(\xi,Z,R_1,\cdots,R_n)>0$
such that whenever $N\geq N_0$ is sufficiently large, 
$$
\mathbb{P}(\sigma_{\text{min}}(\mathfrak{D}_a)\leq N^{-D_Z})\leq C_{q,\xi}N^{-q}.
$$
\end{Proposition}

\begin{proof}
    We always work on the high probability event $\Omega_{\text{high-prob}}$. Then the only place in the Gaussian proof requiring a modification is the use of Gaussian anti-concentration in \eqref{lines1186}. For $\xi$ having a planar density $\|f\|_{L^\infty(\mathbb{C})}\leq K$, we have the following replacement: let $\mathcal{G}_k$ be the random vector defined in \eqref{gaussiangk} but where we replace Gaussian entry distribution by the distribution $\xi$. Let $\Gamma$ be the covariance matrix $\operatorname{Cov}(\mathcal{G}_k)$ as in \eqref{covgk} and recall that $Y_\sigma(E)$ is a $\dim Y_\sigma(E)$-dimensional complex subspace where $\Gamma_{Y_\sigma(E)}\geq \sigma' I_{Y_\sigma(E)}$ for $\sigma'=\frac{\sigma^2}{2}N^{-D_d-1}$, then via the bounded density projection estimate in Corollary \ref{corollary6.8}, 
$$
\mathbb{P}(\|P_{Y_\sigma(E)}\mathcal{G}_k-P_{Y_\sigma(E)}B_k'\|_{\text{HS}}\leq r)\leq (CK\frac{r^2}{\sigma'})^{\dim Y_\sigma(E)}
$$for a universal constant $C>0.$ The rest of the proof follows as in the Gaussian case.
\end{proof}

Then we can generalize Theorem \ref{theorem1290} to the non-Gaussian case:

\begin{theorem}\label{theorems6.30}Let $\xi$ be a complex random variable with mean $0$, variance $1$, finite moments of all orders and admitting a complex density $f$ satisfying $\|f\|_{L^\infty(\mathbb{C})}\leq K<\infty$.
Fix integers $m$ and $n$, and $n$ matrices $R_1,\cdots,R_n\in M_m(\mathbb{C})$. Denote by $\mathcal{T}_N$ the following matrix in $M_{mN}(\mathbb{C})$, where $X_1^N,\cdots,X_n^N$ are i.i.d. $N\times N$ matrices with i.i.d. entries of distribution $\frac{1}{\sqrt{N}}\xi$:
$$
\mathcal{T}_N:=I_m\otimes I_N+\sum_{\ell=1}^nR_\ell\otimes X_\ell^N
.$$
Then for any $q\in\mathbb{N}_+$ we can find constants $C^{R,\xi},c^{R,\xi},c_0^{R,\xi}$ depending on $n,m,R_1,\cdots,R_n$ and the law of $\xi$, as well as a constant $C_{q,\xi}>0$ depending only on $q,m,n$ and $\xi$, such that, whenever $N\geq N_0=N_0(\xi,R_1,\cdots,R_n)$,
 $$
\mathbb{P}\{\sigma_{min}(\mathcal{T}_N)\leq\epsilon
\}\leq C^{R,\xi}(N^{c_0^{R,\xi}}\epsilon^{c^{R,\xi}})+C_{q,\xi}N^{-q}.
    $$    Moreover, $c_0^{R,\xi}$ and $c^{R,\xi}$ can be chosen to depend only on $m$ and $n$.
\end{theorem}

\begin{proof}We again first reduce to an irreducible module $Z$ under left $R_1,\cdots,R_n$-action, and denote the action again by $R_1,\cdots,R_n$.

Compared to the proof of Theorem \ref{theorem1290} in the Gaussian case, the only two places where Gaussian law is used are (1) estimate \eqref{determinantmeans} where we use Gaussian chaos, and (2) estimate \eqref{cerberywrights} where we use Carbery-Wright \cite{carbery2001distributional}. We substitute these two steps by non-Gaussian estimates. But the estimates we derive are not dimension free. Thus we show that $\det D_a(\zeta)$, being a polynomial of degree $\dim Z$ in $(\dim Z)^2$ entries, can indeed be rewritten as a polynomial depending on an $n(\dim Z)^2\leq nm^2$-dimensional random vector which has a bounded density (rather than a polynomial depending on all the $N$ random variables.) 

Fix $\mathcal{T}_{-a}$ and
denote by 
$$D_a(\zeta)=D_{a,0}+\frac{1}{\sqrt{N}}\sum_{\ell=1}^n\sum_{i=1}^NC_{\ell i}\zeta_{\ell i},$$ where $\zeta_{\ell i}$ are i.i.d. copies of $\xi$ and $D_{a,0},C_{\ell i}\in M_m(\mathbb{C})$ are fixed by the unrevealed columns $\mathcal{T}_{-a}$. Denote by $V_\ell=\operatorname{Span}(C_{\ell 1},\cdots,C_{\ell N})$ and assume that $\dim V_\ell =r_\ell$. We now choose an orthonormal basis $F^\ell _1,\cdots,F_r^\ell$ for $V_\ell$ and write 
$$C_{\ell i}=\sum_{\alpha=1}^{r_\ell} f^\ell _{\alpha i}F^\ell _\alpha.
$$
Then 
$$
D_a(\zeta)=D_{a,0}+\frac{1}{\sqrt{N}}\sum_{\ell=1}^n\sum_{\alpha=1}^{r_\ell} Y_\alpha^\ell  F_\alpha^\ell,\quad Y_\alpha^\ell=\sum_{i=1}^N f_{\alpha i}^\ell \zeta_{\ell i}
.$$Equivalently, we can write 
$$
Y^\ell =F^\ell \zeta^\ell,\quad F^\ell =(f^\ell _{\alpha i})_{\alpha\leq r_\ell ,i\leq N}\in\mathbb{C}^{r_\ell \times N},\quad \zeta^\ell=(\zeta_{\ell 1},\cdots,\zeta_{\ell N})^T.
$$Then $F^\ell $ has rank $r_\ell$. Denote by $\Sigma_{F^\ell}=F^\ell (F^\ell)^*$, then $\Sigma_{F^\ell}$ is positive definite. Denote by 
 $U_{F^\ell}:=\Sigma_{F^\ell}^{-1/2}F_\ell$. Then $U_{F^\ell}U_{F^\ell}^*=I_{r_\ell}$ and thus 
$$D_a(\zeta)=D_{a,0}+\frac{1}{\sqrt{N}}\sum_{\ell=1}^n\sum_{\alpha=1}^{r_\ell}(\Sigma_{F^\ell}^{1/2}U_{F^\ell}\zeta^\ell)_\alpha F_\alpha^\ell.$$ 

For the first usage of Gaussians, note that similarly as in  \eqref{determinantmeans} we have that for $N$ sufficiently large and with probability at least $1-C_{q,\xi}N^{-q}$ over $\mathcal{T}_{-a}$, the same frame lower bound as in \eqref{covariancelowerbound} holds. Then we have the following determinant estimate
\begin{equation}\label{determinantmeansnew}
\mathbb{E}_{\zeta\sim \mathcal{N}_\mathbb{C}(0,1)}[|\det D_a(\zeta)|^2]\geq C_{Z} N^{-B_Z\dim Z},
\end{equation}where $\zeta$ takes the law of a standard complex Gaussian. For this step, we regard $\det D_a(\zeta)$ as a degree-$\dim Z$ polynomial depending on $nN$ independent random variables $\zeta_{\ell i}$ rather than using the projected density $U_{F^\ell}\zeta^\ell$. Then by Lemma \ref{oftheorem3.1} (the Gram matrix is positive definite because $\xi$ has a bounded density), we can find a constant $C_{Z,m,n,\xi}$ which further depends on $m,n$ and the law of $\xi$ so that 
$$
\mathbb{E}_{\zeta\sim \xi}[|\det D_a(\zeta)|^2]\geq C_{Z,m,n,\xi} N^{-B_Z\dim Z-\dim Z}
,$$
where the $N^{-\dim Z}$ factor comes from Lemma \ref{oftheorem3.1} with $s=nN,d=\dim Z\leq m$.

For the second usage of Gaussians, we will use Proposition \ref{proposition6.2} instead, which provides a substitute for Carbery-Wright in the non-log-concave setting. By Fact \ref{facts1518}, $U_{F^\ell}\zeta^\ell$ has (uniformly in $\ell$) bounded moments of all orders. Then we can apply Proposition \ref{proposition6.2} to $\det D_a(\zeta)$ as a function depending on the random vectors $U_{F^\ell}\zeta^\ell\in\mathbb{C}^{r\ell}$ since for each $\ell$, the density of $U_{F^\ell}\zeta^\ell$ is bounded by $(eK)^{r_\ell}$. We get, for large enough $N\geq N_0$, for any $t>0$:
$$\mathbb{P}_{\zeta\sim\xi}(|\det D_a(\zeta)|\leq t)\leq C_Z'(tN^{A_z})^{\theta_Z
}$$ for constants $A_Z,\theta_Z$ depending only on $\dim Z$, and the constant $C_Z'$ also depends on $\xi$.
The exponent in $\epsilon$ given by Proposition \ref{proposition6.2} is weaker than the Gaussian case but still sufficient for the proof. Here $N_0=N_0(\xi,R_1,\cdots,R_n)$.

Then we obtain least singular value lower bound for $D_a(\zeta)$ via the elementary inequality $|\det D_a(\zeta)|\leq\sigma_{\text{min}}(D_a(\zeta))\|D_a(\zeta)\|^{m-1}$, and that on $\Omega_{\text{high-prob}}$ $\|D_a(\zeta)\|$ is polynomially bounded.
All other steps of the proof of Theorem \ref{theorem1290} follow without change, and the non-Gaussian version of the estimate for $\sigma_{\text{min}}(\mathcal{T}_N)$ is proven.
\end{proof}

We can then generalize Theorem \ref{theoremonline89} to non-Gaussian entries via linearization:

\begin{theorem}
    Let $n\in\mathbb{N}$ and $\mathfrak{p}\in\mathbb{C}\langle x_1,\cdots,x_n\rangle$. Let $\xi$ have the distribution stated in Theorem \ref{nongaussiantheorem}. Then there are constants $c_0(\mathfrak{p},\xi)$ and $c(\mathfrak{p},\xi)$ depending on $\mathfrak{p}$ and $\xi$; a constant $\widehat{C}(\mathfrak{p},\xi,z)>0$ depending on $\mathfrak{p}$, $\xi$ and $z$; as well as for any $q\in\mathbb{N}_+$ a constant $C_q=C_{\mathfrak{p},q,\xi}>0$ depending on $\mathfrak{p},q,\xi$, so that we have the following estimate. Let $X_1^N,\cdots,X_n^N$ be independent $N\times N$ complex matrices with i.i.d. entries of law $\frac{1}{\sqrt{N}}\xi$; and denote by $P^N=\mathfrak{p}(X_1^N,\cdots,X_n^N)$. Then for any $z\in\mathbb{C},z\neq \mathfrak{p}(0)$ we can find some fixed $N_0=N_0(z,\mathfrak{p},\xi)$ such that whenever $N\geq N_0$, we have for any $\epsilon>0$,
    $$
\mathbb{P}\{\sigma_{min}(P^N-z)\leq\epsilon
\}\leq \widehat{C}(N^{c_0}\epsilon^c)+C_qN^{-q}.
    $$
\end{theorem}

The proof is immediate by combining Theorem \ref{theorems6.30} with Lemma \ref{lemma2.222} and Fact \ref{fact2.568}.

We have used the following comparison lemma for polynomial moments:
\begin{lemma}\label{oftheorem3.1}
    Let $\mu$ be a complex random variable and fix an integer $d>0$. Suppose that 
    $$
\mathbb{E}|\mu|^{2d}<\infty,
    $$and the following Gram matrix 
    $$
M^\mu_{(d)}=(\mathbb{E}[\mu^a\bar{\mu}^b])_{0\leq a,b\leq d}
    $$is positive definite. Let $\xi_1,\cdots,\xi_s$ be i.i.d. copies of $\mu$, and let $g_1,\cdots,g_s$ be i.i.d. standard complex Gaussians. Then for any holomorphic polynomial $$Q(z_1,\cdots,z_s)$$
    with total degree at most $d$, we have the comparison
    \begin{equation}
C_{d,\mu}^{-1}(1+s)^{-d}\mathbb{E}|Q(g_1,\cdots,g_s)|^2\leq\mathbb{E}|Q(\xi_1,\cdots,\xi_s)|^2\leq C_{d,\mu}(1+s)^d\mathbb{E}|Q(g_1,\cdots,g_s)|^2       .
    \end{equation}
    Here the constant $C_{d,\mu}$ depends only on $d$ and the law of $\mu$, but does not depend on $s$.
\end{lemma}
The proof of Lemma \ref{oftheorem3.1} is deferred to Appendix \ref{AppendixA}. The positive definiteness of $M^\mu_{(d)}$ is equivalent to the fact that there does not exist a nonzero holomorphic polynomial $q\in\mathbb{C}[z],\deg q\leq d$ such that $q(\xi)=0$ a.s. Indeed, take any polynomial $q(z)=\sum_{a=0}^d c_az_a$, then 
$$
\mathbb{E}|q(\xi)|^2=c^*M^\mu_{(d)}c.
$$
So $q(\xi)=0$ a.s. if and only if $c^*M^\mu_{(d)}c=0$. When $\mu$ has a density with respect to planar Lebesgue measure, then clearly such $q$ does not exist, and thus $M^\mu_{(d)}$ is positive definite.

We have also used the following (weaker) substitute for Carbery-Wright (Theorem \ref{theoremlines102}), as the latter holds for all log-concave distributions but we do not assume log concavity here:
\begin{Proposition}\label{proposition6.2}
  Consider a random vector $\mathcal{Z}\in\mathbb{C}^n$ having a joint density $f_\mathcal{Z}$ on $\mathbb{C}^n$ satisfying $\|f_\mathcal{Z}\|_{L^\infty(\mathbb{C}^n)}\leq K$. Assume that for some $Q>2D$, $\mathbb{E}|\mathcal{Z}|^Q\leq M_Q<\infty$. Take $P:\mathbb{C}^n\to\mathbb{C}$ a holomorphic polynomial with total degree at most $D$. If $P\neq 0$, then for every $0<\epsilon<1$,
 $$
\mathbb{P}(|P(\mathcal{Z})|\leq\epsilon\|P(\mathcal{Z})\|_{L^2})\leq C_{n,D,Q,K,M_Q}\epsilon^\frac{Q}{D(Q+2n)}.
 $$\end{Proposition}
 
 We also need the following estimate, which is a complex version of \cite{rudelson2015small}, Theorem 1.1.

 \begin{lemma}\label{lemma1519} Let $X=(\xi_1,\cdots,\xi_N)\in\mathbb{C}^N$ have independent coordinates and each $\xi_i$ has a Lebesgue density $f_i$ satisfying $\|f_i\|_{L^\infty(\mathbb{C})}\leq K$.
     Consider $U:\mathbb{C}^N\to\mathbb{C}^r$ an isometry onto its image: $UU^*=I_{\mathbb{C}^r}$. Then $UX$ has a planar density $P_U$ on $\mathbb{C}^r$ satisfying 
     $$
\|P_U\|_{L^\infty(\mathbb{C}^r)}\leq (eK)^r.
     $$
 \end{lemma}

 \begin{corollary}\label{corollary6.8}In the setting of Lemma \ref{lemma1519}, let $A:\mathbb{C}^N\to\mathbb{C}^d$ be a surjective complex linear map such that $AA^*\geq\sigma I_{\mathbb{C}^d}$. Then $AX$ has a planar density on $\mathbb{C}^d$ bounded by $(\frac{eK}{\sigma})^d$. Consequently, for any $y\in\mathbb{C}^d$ and $r>0$,
 $$
\mathbb{P}\{\|AX-y\|_2\leq r\}\leq (\frac{CKr^2}{\sigma})^d.
 $$ 
 \end{corollary}

 \begin{fact}\label{facts1518}
     Let $X=(\xi_1,\cdots,\xi_N)\in\mathbb{C}^N$ have independent mean zero entries with 
     $$
\mathbb{E}[|\xi_i|^2]=1,\quad \mathbb{E}[|\xi_i|^Q]\leq M_Q
     $$ for some $Q\geq 2$. Consider an orthogonal projection $P$ on $\mathbb{C}^N$ of rank $k$, then the projected vector $\mathcal{Z}:=PX$ has $Q$-moments bounded by a constant depending on $k,Q,M_Q$ but not $N$. 
 \end{fact}

Proposition \ref{proposition6.2}, Lemma \ref{lemma1519}, Corollary \ref{corollary6.8} and Fact \ref{facts1518} are proven in Appendix \ref{AppendixA}.
 
\subsection{Proof of Brown measure convergence: non-Gaussian case}In this final section we generalize the argument in Section \ref{section5555} to the non-Gaussian case, thus completing the proof of Theorem \ref{nongaussiantheorem}. We again split the proof into Steps \ref{step11}, \ref{step22} and \ref{step33}.

For Step \ref{step11}, since the $X_i^N$ have i.i.d. entries of law $\frac{1}{\sqrt{N}}\xi$ where $\xi$ has mean 0, variance 1 and finite moments of all orders, the family $(X_1^N,\cdots,X_n^N)$ converges in $*$-moments in probability to the free circular family $(c_1,\cdots,c_n)$. This can be checked via the standard trace moment expansion since the only leading index configurations are given by the tree pairings which match each entry to its conjugate, and the contributions involving $\mathbb{E}\xi^2$ have a lower order. This implies, for any $k\geq 1$,
$$
\operatorname{tr}_N[(zI_N-P_N)(zI_N-P_N)^*]^k\to\tau(|z-\mathfrak{p}(c_1,\cdots,c_n)|^{2k})
$$
and thus by moment tightness, $\nu_N^z$ converges weakly in probability towards a measure $\nu^z=\operatorname{Law}(|z-\mathfrak{p}(c_1,\cdots,c_n)|^2)$.

For the proof of Step \ref{step33}, we only use (a) uniform $L^2$ integrability \eqref{ell2integrability}; (b) pointwise convergence by Step \ref{step22}; and (c) boundedness of $\|P^N\|$ which holds on $\Omega_{\text{high-prob}}$. Thus all the steps follow once Step \ref{step22} is proved.

For Step \ref{step22}, we import estimate \eqref{bythegaussianmodel} in the Gaussian case from \cite{cook2022spectrum}, Lemma 8.1, and the remaining arguments do not depend on the entry law. Thus, the only missing component is a non-Gaussian version of \eqref{bythegaussianmodel}, which is stated as the following Lemma \ref{lemmaline1522}. Having proven this, the Brown measure convergence in Theorem \ref{nongaussiantheorem} is fully proven.

\begin{lemma}\label{lemmaline1522}
Let $\deg\mathfrak{p}\geq 1$. For the non-Gaussian model, for each $z\in\mathbb{C}$ we can find positive constants $c_1',C',c_2'\in(0,1),c_3'$ depending on $\mathfrak{p},\xi,z$ such that for any $x\in[N^{-c_1'},c_3']$,
  \begin{equation}
      \mathbb{E}[\nu_N^z([0,x])]\leq 2C'x^{1-c_2'}.
  \end{equation}
\end{lemma}

\begin{proof}
  Since the entry law $\xi$ has finite moments of all orders, by a standard truncation theorem we can find a random variable $\xi_N$ having mean 0, variance 1 and bounded by $N^{0.4}$ a.s., such that if we denote by $Y_i^N$ an $N\times N$ matrix with i.i.d. law $\frac{1}{\sqrt{N}}\xi_N$, then $\|X_i^N-Y_i^N\|\leq N^{-100}$ with probability at least $1-CN^{-100}$. The truncation error can be absorbed in the final bound and we assume $|\xi|\leq N^{0.4}$ a.s. in the proof.

  Let $$A_{\xi,N}^z:=P^N-z,\quad \mathcal{L}^{\xi,N}_z=\mathcal{L}_z(X_1^\xi,\cdots,X_n^\xi), $$
where we use the shorthand $\xi=G$ to denote the complex Gaussian case and otherwise the superscript $\xi$ is the entry law of $X_i$ in Theorem \ref{nongaussiantheorem}. Here $\mathcal{L}^{\xi,N}_z$ is a shorthand for the linearization matrix constructed in \eqref{equation325}. We take the Hermitization
$$
\mathcal{H}_{\xi,N}^z=\begin{bmatrix}
    0&\mathcal{L}^{\xi,N}_z\\(\mathcal{L}^{\xi,N}_z)^*&0
\end{bmatrix}.
$$

On $\Omega_{\text{high-prob}}$, by Schur complement elimination as executed in the proof of Lemma \ref{lemma2.222}, we can find $E(X),F(X)$ such that
\begin{equation}\label{lines15470}
E(X)\mathcal{L}_z^{\xi,N}(X)F(X)=\begin{bmatrix}
    A_{\xi,N}^z&0\\0&I_{(m-1)N}
\end{bmatrix}
,\end{equation}
and that on the norm bounded event $\Omega_{\text{high-prob}}$ we have
\begin{equation}\label{lines1553}
\|E^{\pm 1}\|+\|F^{\pm 1}\|\leq C_{\mathfrak{p},z}
\end{equation}for a constant depending on $\mathfrak{p}$ and $z$. Indeed, the proof of Lemma \ref{lemma2.222} shows that $E(X)$ and $F(X)$ consist of $B_q(X),C_q(X)$ and $D_q(X)^{-1}$ in \eqref{schurcomplements}, and since $D_q$ is upper triangular, the norm of these three matrices are all polynomially bounded in the norm of $(X_i^\xi)_{i=1}^n$. This verifies \eqref{lines1553}.

Denote by $s_i(\cdot)$ the singular value of the matrix, then combining \eqref{lines15470} and \eqref{lines1553}, we deduce that we can find $\mathfrak{p},z$-dependent constants $C_1,C_2$ such that for $t$ sufficiently small, 
\begin{equation}
\label{lines15590}
\operatorname{Card}\{i:s_i(A_{\xi,N}^z)\leq t\}\leq \operatorname{Card}\{i:s_i(\mathcal{L}_z^{\xi,N})\leq C_1t\},  
\end{equation}
\begin{equation}\label{whatisnewline1701}
\operatorname{Card}\{i:s_i(\mathcal{L}_z^{\xi,N})\leq t\}  \leq \operatorname{Card}\{i:s_i(A_{\xi,N}^z)\leq C_2t\}.  
\end{equation}
For the Gaussian model, by \eqref{bythegaussianmodel} (which follows from \cite{cook2022spectrum}, Lemma 8.1), for any $x\in[N^{-c_1},1]$ we have 
$$
\mathbb{E}_G[\nu_N^z([0,x])]\leq Cx^{1-c_2},\text{ which can be written as }\mathbb{E}_{\nu_{A_{G,N}^z(A_{G,N}^z)^*}}([0,x])\leq Cx^{1-c_2}.
$$We apply this in conjunction with \eqref{whatisnewline1701} and get 
$$
\mathbb{E}_{\mu_{\mathcal{H}_{G,N}^z}}([-t,t])\leq C_3t^{2(1-c_2)}.
$$ This is an estimate for the empirical spectral measure of $\mathcal{H}_{G,N}^z$. Denote by $m_{\xi,N}^z(w):=\mathbb{E}\operatorname{tr}_{2mN}((wI-\mathcal{H}_{\xi,N}^z)^{-1})$ as the Stieltjes transform, and similarly $m_{G,N}^z$ for the Gaussian model. We can then use an elementary dyadic estimate and get the following Stieltjes transform estimate:
\begin{equation}\label{transferral}
|\Im m_{G,N}^z(i\eta)|=\mathbb{E}\int_\mathbb{R}\frac{\eta}{\lambda^2+\eta^2}
d\mu_{\mathcal{H}_{G,N}^z}(\lambda)\leq C_4\eta^{-q},\quad N^{-c_4}\leq \eta\leq 1,\end{equation} for some $q<1$ and some $c_4\in(0,1)$.

Now we compare the Stieltjes transform of $\mathcal{H}_{G,N}^z$ and $\mathcal{H}_{\xi,N}^z$ via \cite[Theorem 2.10]{brailovskaya2024universality}. \footnote{\cite{brailovskaya2024universality}, Theorem 2.10 considers a general random matrix of the form $X=\sum_i Z_i$ where the $Z_i$ are mutually independent, and associates to $X$ a Gaussian random matrix $G$ such that $\mathbb{E}[X]=\mathbb{E}[G],\mathbb{E}[XX^*]=\mathbb{E}[GG^*]$. The comparison of resolvents of $X$ and $G$ in Theorem 2.10 involves two parameters: one is $\mathbb{E}[(X-\mathbb{E}X)^2]$ which is $O(1)$ in our case, and the other is $\max_i\|Z_i\|$ which is at most $N^{-0.1}$ by our truncation.} Temporarily assume that $\mathbb{E}[\xi^2]=0$, then $\mathcal{H}_{G,N}^z$ and  $\mathcal{H}_{\xi,N}^z$ have the same mean and variance profile in the sense of \cite{brailovskaya2024universality}, and by Theorem 2.10 therein, we have
$$
\|\mathbb{E}[(wI-\mathcal{H}_{\xi,N}^z)^{-1}]-\mathbb{E}[(wI-\mathcal{H}_{G,N}^z)^{-1}]\|\lesssim \frac{C_5N^{-c_6}}{(\Im w)^4}
$$for some $C_5,c_6>0$, where we use our truncation step so that each entry of $Y_i^N$ is a.s. bounded by $N^{-0.1}$. We then take the trace and obtain the Stieltjes transform comparison. When $\mathbb{E}[\xi^2]=\rho\neq 0$, we denote by $\mathcal{H}_{G_\rho,N}^z$ the matrix obtained by replacing every circular Gaussian in $\mathcal{H}_{G,N}^z$ with the Gaussian having pseudovariance $\rho$. Then the same estimate holds where we replace the $\mathcal{H}_{G,N}^z$ by  $\mathcal{H}_{G_\rho,N}^z$ because the latter has the same variance profile as $\mathcal{H}_{\xi,N}^z$.
Then we use Lemma \ref{lemmaline1725} to compare the resolvent of $\mathcal{H}_{G,N}^z$ and $\mathcal{H}_{G_\rho,N}^z$ and the difference is of order $N^{-1}|\Im w|^{-3}$.  Combining all the above cases, we can find some $c_8\in(0,c_4),c_9\in[q,1)$ such that whenever $\eta\geq N^{-c_8}$,
$$
|m_{\xi,N}^z(i\eta)-m_{G,N}^z(i\eta)|\leq\eta^{-c_9},
$$and thus \eqref{transferral} implies that 
$$
|\Im m_{\xi,N}^z(i\eta)|\leq \eta^{-c_9},\quad N^{-c_8}\leq\eta\leq 1.
$$Since 
$$
\mu_{\mathcal{H}_{\xi,N}^z([-t,t])}\leq 2t|\Im m_{\xi,N}^z(it)|,
$$and using the comparison \eqref{lines15590}, we get
$$
\mathbb{E}_{\nu_{A_{\xi,N}^z(A_{\xi,N}^z)^*}}([0,x])\leq C_{10}x^{(1-c_9)/2},\quad x\in[ N^{-2c_8},c_{11}],
$$for some fixed $c_{11}>0.$ This completes the proof.\end{proof}

We have used the following perturbative estimate, see Appendix \ref{AppendixA} for the proof.

\begin{lemma}\label{lemmaline1725}
    Let $\rho\in\mathbb{C}:|\rho|\leq 1$ and let $\mathcal{H}_{\rho,N}^z$ be the Hermitized linearization using complex Gaussian entries $g_{\ell ab}^{(\rho)}$ satisfying 
    $
\mathbb{E}|g_{\ell ab}^{(\rho)}|^2=1,\quad \mathbb{E}(g_{\ell ab}^{(\rho)})^2=\rho.
    $ Then let $\mathcal{H}_{0,N}^z$ be the circular Gaussian case, i.e., $\rho=0$. Then for any $w\in\mathbb{C}:\Im w>0$,
    $$
\|\mathbb{E}(wI-\mathcal{H}_{\rho,N}^z)^{-1}-\mathbb{E}(wI-\mathcal{H}_{0,N}^z)^{-1}\|\leq \frac{C_\mathfrak{p}|\rho|}{N(\Im w)^3},
    $$and thus 
    $$
|\mathbb{E}\operatorname{tr}_{2mN}(wI-\mathcal{H}_{\rho,N}^z)^{-1}-\mathbb{E}\operatorname{tr}_{2mN}(wI-\mathcal{H}_{0,N}^z)^{-1}|\leq \frac{C_\mathfrak{p}|\rho|}{N(\Im w)^3}.
    $$
\end{lemma}

\appendix

\section{Proofs of auxiliary results}
\label{AppendixA}This appendix contains the proofs of some technical results that are left unproven in the main text.

\begin{lemma}\label{howdoweuse681} Let $Z$ be a finite-dimensional Hilbert space.
    Every left ideal $E\subset\operatorname{End}(Z)$ has the form $E=\operatorname{End}(Z)P$ for a projection $P$ on $Z$.
\end{lemma}

\begin{proof}[\proofname\ of Lemma \ref{howdoweuse681}]
    For a left ideal $E$, define its common kernel as
    $$
K:=\cap_{A\in E}\ker A\subset Z.
    $$Thus each $A\in E$ vanishes on $K$ and we have
    $$
E\subset\{T\in\operatorname{End}(Z):K\subset\ker T\}.
    $$We next check that the reverse inclusion also holds. 

    As $Z$ is finite dimensional, we can choose finitely many elements $e_1,\cdots,e_r\in E$ such that 
    $$\cap_{\alpha=1}^r\ker e_\alpha=K.
$$Then define the following map 
$$
\mathcal{E}:Z\to Z^r, \quad\mathcal{E}(x)=(e_1x,\cdots,e_rx).
$$Then we must have $\ker\mathcal{E}=K$.

Next, we take any linear map 
$$T:Z\to Z
$$ such that 
$$
K\subset\ker T. 
$$Since $\ker \mathcal{E}=K$, the map $T$ factors through $E$, which means that there exists a linear map 
$$
L:\operatorname{Ran}\mathcal{E}\to Z
$$satisfying that 
$$
T=L\mathcal{E}.
$$Next we extend $L$ arbitrarily into some linear map 
$$
\widetilde{L}:Z^r\to Z,
$$where given in coordinates it has the form 
$$
\widetilde{L}(y_1,\cdots,y_r)=\sum_{\alpha=1}^r L_\alpha y_\alpha,
$$
for some $L_\alpha\in\operatorname{End}(Z).$ Then we must have
$$
T=\widetilde{L}\mathcal{E}=\sum_{\alpha=1}^r L_\alpha e_\alpha.
$$Since $e_\alpha\in E$ and $E$ is left ideal, we have $L_\alpha e_\alpha \in E$ and $T\in E.$

    Finally, we now choose a projection 
    $$
P:Z\to Z,
    $$with $\ker P=K$. For instance we can take the orthogonal projection onto $K^\perp$. Then we get the equality 
    $$
\operatorname{End}(Z)P=\{LP:L\in\operatorname{End}(Z)\}=\{T\in \operatorname{End}(Z):K\subset\ker T\}.
    $$
\end{proof}

\begin{proof}[\proofname\ of Lemma \ref{blocksecondidentity}]Fix a block $a$ and work on the decomposition $E=E_a\oplus E_{-a}$. For this decomposition, the Gram operator $G:=T^*T$ has a block form $$G=\begin{bmatrix}
    A&B\\B^*&C
\end{bmatrix},
    $$where each block has the following entries
    $$
A=T_a^*T_a,\quad B=T_a^*T_{-a},\quad C=T_{-a}^*T_{-a}.
    $$Then $T_{-a}$ is injective since $T$ is injective, so $C=T_{-a}^*T_{-a}$ is invertible.

The orthogonal projection onto $\operatorname{Ran}(T_{-a})^\perp=\ker(T_{-a}^*)$ has the form 
$$
P_a=I_{Z^N}-T_{-a}(T_{-a}^*T_{-a})^{-1}T_{-a}^*=I_{Z^N}-T_{-a}C^{-1}T_{-a}^*.
$$Therefore, we must have
$$
\mathfrak{Exp}_a^*\mathfrak{Exp}_a=T_a^*P_aT_a=T_a^*T_a-T_a^*T_{-a}C^{-1}T_{-a}^*T_a=A-BC^{-1}B^*,
$$and the latter is exactly the Schur complement of $C$ in $G$. Therefore, by the Schur block inverse formula, we get that the $E_a\to E_a$ diagonal block of $G^{-1}$ is precisely given by 
$$
(A-BC^{-1}B^*)^{-1}=(\mathfrak{Exp}_a^*\mathfrak{Exp}_a)^{-1}.
$$Taking the trace, we have
$$
\operatorname{Tr}(G^{-1})=\sum_{a=1}^N\operatorname{Tr}((\mathfrak{Exp}_a^*\mathfrak{Exp}_a)^{-1}).
$$Finally, since $G=T^*T$, we have
$$\operatorname{Tr}(G^{-1})=\sum_{s=1}^{N\dim Z}\sigma_s(T)^{-2}.$$ This verifies the identity.
\end{proof}

\begin{proof}[\proofname\ of Fact \ref{facts3.33}] In this proof we simply write $m=\dim Z$.

    The matrix $\mathcal{T}_{-a}$ has $mN$ rows and $m(N-1)$ columns and thus its maximal possible rank is $m(N-1)$. To prove it has full rank almost surely, we now exhibit one $m(N-1)\times m(N-1)$ minor which is not identically zero.

    Choose block rows $k\in [N]\setminus\{a\}$ and all remaining block columns $i\in [N]\setminus\{a\}$, we get the square block submatrix 
    $$
M(G)=[\delta_{ki}I_Z+\sum_{\ell=1}^n G_\ell(k,i)R_\ell]_{k,i\neq a}.
    $$
    The determinant is a polynomial in the Gaussian entries $G_\ell(k,i)$ as every entry of $M(G)$ is affine linear in all these Gaussian variables.

Evaluating at the zero matrix $G_\ell(k,i)=0$, we have 
$$
M(0)=[\delta_{ki}I_Z]_{k,i\neq a}=I_{m(N-1)}.
$$Thus $\det M(G)$ is not identically zero. Since the Gaussian entries have a continuous density, a nonzero polynomial in these entries will vanish only with probability zero. Therefore, almost surely we have 
$$\det M(G)\neq 0,$$
    and therefore we have 
    $$
\operatorname{rank}(\mathcal{T}_{-a})=m(N-1),\quad\dim W_a=m. 
    $$Via the same steps we can show that $\mathcal{T}_N$ is invertible almost surely.
\end{proof}

\begin{proof}[\proofname\ of Lemma \ref{lemma4.20}]
    The map $P_{Y\mid _C}:C\to Y$ has least singular value at least $\sqrt{1-\tau^2}$. Since $\dim C=\dim Y$, the injective map $P_{Y\mid C}$ is onto. Therefore, for any $y\in Y$, we can find some $x\in C$ satisfying that
    $$
P_Yx=y,\quad \|x\|\leq (1-\tau^2)^{-1/2}\|y\|.
    $$Then since $y-x=-P_{Y^\perp}x$ and $x\in C$, 
    $$
d(y,C)\leq\|y-x\|=\|P_{Y^\perp}x\|\leq\tau\|x\|\leq\frac{\tau}{\sqrt{1-\tau^2}}\|y\|.
    $$This completes the proof.
\end{proof}

\begin{proof}[\proofname\ of estimate \eqref{thisisfromref3}, recalled from \cite{cook2022spectrum}]
Recall that we can decompose a Ginibre matrix as the sum of two independent GUE matrices via $X_j^N=(Y_j^N+iZ_j^N)/\sqrt{2}$ for two independent GUE matrices $Y_j,Z_j$. Then we can rewrite $(z-P^N)(z-P^N)^*$ as a polynomial of independent GUE matrices, so that by \cite[Lemma 5.5.4]{anderson2010introduction}, we can find $c_1>0$ such that for any $\Im\zeta\in[N^{-c_1},1]$ and all sufficiently large $N$,
$$|g^z(\zeta)-g^z_N(\zeta)|\leq\frac{c_2}{N^2(\Im\zeta)^{c_3}}.$$
We then only need to prove that for some $q>-1,K>0$ we have 
$$
|\Im g^z(i\epsilon)\leq K\epsilon^q|
.$$ First, by \cite[Corollary 1.2]{shlyakhtenko2015freely}, we verify that $\nu^z$ has no atoms, and by \cite[Theorem 1.1]{shlyakhtenko2015freely}, $g^z$ is bounded near the real line excluding a discrete set $A$. If $A$ does not contain origin we take $q=0$. If $A$ contains the origin, the same result shows $g^z(\zeta)\sim K\zeta^q$ for some $K>0,q\in\mathbb{Q}$, while no atoms requires $q>-1$. Clearly, all these steps work for $\mathfrak{p}$ of any degree.
\end{proof}

\begin{proof}[\proofname\ of Lemma \ref{oftheorem3.1}]
    Denote by 
    $$
\mathcal{P}_{\leq d}^{(s)}=\{Q\in\mathbb{C}[z_1,\cdots,z_s]:\deg Q\leq d
\}. $$ Its dimension is at most 
$$
M_s=\binom{s+d}{d}\leq C_d(1+s)^d.
$$
For the upper bound, consider a multiindex $\alpha=(\alpha_1,\cdots,\alpha_s),|\alpha|\leq d$, we write 
$$z^\alpha=z_1^{\alpha_1}\cdots z_s^{\alpha_s},\quad \alpha!=\alpha_1!\cdots\alpha_s!.
$$
For the standard complex Gaussian, the normalized monomials $$e_\alpha(z):=\frac{z^\alpha}{\sqrt{\alpha!}}$$ form an orthonormal basis of $\mathcal{P}_{\leq d}^{(s)}$ in $L^2(\gamma^s)$, where $\gamma$ is the standard complex Gaussian. Therefore if we write any polynomial $Q$ in this base:
$$
Q=\sum_{|\alpha|\leq d}c_\alpha e_\alpha,
$$
then we have 
$$
\mathbb{E}_g|Q(g)|^2=\sum_{|\alpha|\leq d}|c_\alpha|^2.
$$
Now turn to the $\mu^s$-norm. The entries of Gram matrix under this new norm is 
$$
\mathcal{M}_{\alpha\beta}^\mu=\frac{1}{\sqrt{\alpha!\beta!}}\prod_{j=1}^s\mathbb{E}[\xi^{\alpha_j}\bar{\xi}^{\beta_j}].
$$

Since $|\alpha|+|\beta|\leq 2d$, we can find a constant $C_{d,\mu}>0$ such that 
$$ |\mathcal{M}_{\alpha\beta}^\mu|\leq C_{d,\mu}  $$
holds for any value of $s,\alpha,\beta$. Thus we have
$$
\lambda_{\text{max}}(\mathcal{M}^\mu)\leq C_{d,\mu}M_s.
$$Therefore, we have
\begin{equation}\label{justifyupperbound}
\mathbb{E}_\mu |Q(\xi)|^2=c^*\mathcal{M}^\mu c\leq C_{d,\mu}M_s\sum_\alpha |c_\alpha|^2\leq C_{d,\mu}(1+s)^d\mathbb{E}_g|Q(g)|^2
.\end{equation} This justifies the upper bound.

The proof for the lower bound is similar. Indeed, since $M^\mu_{(d)}$ is positive definite by assumption, we apply Gram-Schmidt to the one variable space $\mathbb{C}[z]_{\leq d}$ and get a family of $\mu$-orthonormal polynomials
$$
p_0(z),p_1(z),\cdots,p_d(z),\quad \deg p_k=k.
$$Then for any $|\alpha|\leq d$, $\{p_\alpha\}
$ is an orthonormal basis of $\mathcal{P}_{\leq d}^{(s)}$ in $L^2(\mu^s)$. Therefore, if we write 
$$Q=\sum_{|\alpha|\leq d}a_\alpha p_\alpha,
$$
then 
$$
\mathbb{E}_\mu|Q(\xi)|^2=\sum_{|\alpha|\leq d}|a_\alpha|^2.
$$Now we switch to Gaussian norm estimates, where we also have
$$\mathbb{E}_g[p_\alpha(g)\overline{p_\beta(g)}]=\prod_{j=1}^s\mathbb{E}_g[p_{\alpha_j}(g)\overline{p_{\beta_j}(g)}]
.$$ Again since $|\alpha|+|\beta|\leq 2d$, there are at most $2d$ nontrivial indices. As $p_0,\cdots,p_d$ are polynomials determined by $\mu$, we find $C_{d,\mu}>0$ such that $|\mathbb{E}_g[p_\alpha(g)\overline{p_\beta(g)}]|\leq C_{d,\mu}$ (we can enlarge $C_{d,\mu}$ in \eqref{justifyupperbound} so that both $C_{d,\mu}$ take the same value.) Therefore, the operator norm of the Gaussian Gram matrix under the $\{p_\alpha\}$ basis is at most $C_{d,\mu}M_s$. Therefore,
$$ \mathbb{E}_g|Q(g)|^2\leq C_{d,\mu}M_s\sum_\alpha |a_\alpha|^2=C_{d,\mu}M_s\mathbb{E}_\mu|Q(\xi)|^2.
$$
This justifies the lower bound.\end{proof}

\begin{proof}[\proofname\ of Proposition \ref{proposition6.2}]
     We re-normalize $P(\mathcal{Z})$ and denote by $q(z):=p(z)/\|P(\mathcal{Z})\|_{L^2}$. We write in coefficients 
     $$
q(z)=\sum_{|\alpha|\leq D}c_\alpha z^\alpha,
     $$ and denote by $N(n,D)=\binom{n+D}D$. We define the coefficient norm 
     $$
\|q\|_{\text{coef}}:=\left(\sum_{|\alpha|\leq D}|c_\alpha|^2\right)^\frac{1}{2}.
     $$
     For any $z\in\mathbb{C}^n$, denote by $z^\alpha=z_1^{\alpha_1}\cdots z_n^{\alpha_n}$, then 
     $$
|z^\alpha|\leq\max(1,\|z\|_2^D)\text{ when }|\alpha|\leq D.
     $$Then by Cauchy-Schwarz, 
     $$
|q(z)|\leq \sum_{|\alpha|\leq D}|c_\alpha||z^\alpha|\leq \sqrt{N(n,D)}\|q\|_{\text{coef}}\max(1,\|z\|_2^D).
     $$Taking expectation, 
     $$
1=\mathbb{E}[|q(\mathcal{Z})|^2]\leq N(n,D)\|q\|^2_{\text{coef}}\mathbb{E}\max(1,\|\mathcal{Z}\|_2^{2D}).
     $$Since $Q\geq 2D$, we can upper bound, for some $C_{n,D,Q,M_Q}$ depending on these constants and whose value may change from line to line, 
     $$
\mathbb{E}\max(1,\|\mathcal{Z}\|_2^{2D})\leq C_{n,D,Q,M_Q}
     .$$Then 
     $$
\|q\|_{\text{coef}}\geq c_{n,D,Q,M_Q}>0.
     $$
     Since $q$ is holomorphic, we can apply Cauchy's integration formula to each $c_\alpha$
$$c_\alpha=\frac{1}{(2\pi i)^n}\int_{|\zeta_1|=r}\cdots\int_{|\zeta_n|=r}\frac{q(\zeta)}{\zeta_1^{\alpha_1+1}\cdots\zeta_n^{\alpha_n+1}}
d\zeta_1\cdots d\zeta_n$$
     and deduce that, for each coefficient $\alpha$, 
     $$|c_\alpha|\leq r^{-|\alpha|}\sup_{\mathbb{D}_r^n}|q|,
     $$where $\mathbb{D}_r$ is the disc of radius $r$ in $\mathbb{C}$. In this proof let $B_r$ denote the ball of radius $r$ in $\mathbb{C}^n$ for all $r\geq 0$. Since $(\mathbb{D}_{n^{-1/2}})^n\subset B_1$, we combine with the lower bound on $\|q\|_{\operatorname{coef}}$ to deduce that we can find some $s_*>0$ depending only on $n,D,Q,M_Q$ such that for any $R\geq 1$,
     \begin{equation}\label{line1285}
\sup_{B_R}|q|\geq s_*.
     \end{equation}

We use the following Remez type sublevel set control (see \cite{brudnyui1973extremal} or \cite[Theorem A.1]{rudelson2014invertibility}.) Let $r:\mathbb{R}^m\to\mathbb{R}$ be a real polynomial of degree at most $D$, then for any Euclidean ball $B_R\subset\mathbb{R}^m$ and any $0<t\leq 1$,
$$
\operatorname{Vol}\{z\in B_R:|r(z)|\leq t\sup_{B_R}|r|
\}\leq C_{m,D}t^{1/D}\operatorname{Vol}(B_R).
$$We apply this to $r(z):=|q(z)|^2$ (which has degree $2D$) viewed as a polynomial on $\mathbb{R}^{2n}$, and use \eqref{line1285} to deduce that 
$$
\operatorname{Vol}(\{|q|\leq\epsilon\}\cap B_R)\leq C_{n,D,Q,M_Q}R^{2n}\epsilon^{1/D}.$$
Since the joint density $f_\mathcal{Z}$ of $Z$ satisfies $\|f_\mathcal{Z}\|_\infty\leq K$, the above implies that 
$$
\mathbb{P}(|q(\mathcal{Z})|\leq\epsilon,\mathcal{Z}\in B_R)\leq K\operatorname{Vol}(\{|q|\leq\epsilon\}\cap B_R)\leq C_{n,D,Q,K,M_Q}R^{2n}\epsilon^{1/D}.
$$Finally we control the tail outside: by Markov and the finite $Q$-th moment of $\mathcal{Z}$, 
$$
\mathbb{P}(\|\mathcal{Z}\|_2\geq R)\leq C_{n,Q,M_Q}R^{-Q}.
$$ Choosing $R=\epsilon^{-1/(D(Q+2n))}$, we combine the last two inequalities. This completes the proof.\end{proof}

 \begin{proof}[\proofname\ of Lemma \ref{lemma1519}]Denote by $W=\ker U$, since $UU^*=I_r$ we consider an orthogonal decomposition 
 $$
\mathbb{C}^N=W\oplus \operatorname{Im} U^*,
 $$where $U:\operatorname{Im} U^*\to\mathbb{C}^r$ is an isometry. Then for any $y\in\mathbb{C}^r$, the fiber is given as
 $$
U^{-1}\{y\}=U^*y+W.
 $$Via the coarea formula, the density of $UX$ at $y$ satisfies 
 $$
p_U(y)=\int_W\prod_{i=1}^Nf_i((U^*y+w)_i)dw.
 $$
     Define the coordinate maps $C_i:W\to\mathbb{C},\quad C_iw=w_i$. Denote by $c_i:=\|C_i\|_{\text{op}}^2,$ so that $c_i=1-\|Ue_i\|^2.$ Then the $c_i$ satisfy 
     $$
0\leq c_i\leq 1,\quad \sum_{i=1}^Nc_i=N-r.
     $$
Then for $c_i>0$ define 
$$\widetilde{C}_i=c_i^{-1/2}C_i:\quad W\to\mathbb{C},
$$then they satisfy
$$\widetilde{C}_i\widetilde{C}_i^*=I_\mathbb{C},
$$
$$\sum_{i:c_i>0}c_i\widetilde{C}_i^*\widetilde{C}_i=I_W.
$$Then we set 
$$F_i(z):=f_i(z+(U^*y)_i),$$ so that 
$0\leq F_i\leq  K$ and $\int_\mathbb{C}F_i(z)dz=1$. Since $0\leq c_i\leq 1$, we write 
$$F_i(z)\leq K^{1-c_i}F_i(z)^{c_i},
$$and thus we have 
$$
\prod_{i=1}^N F_i(C_iw)\leq K^{\sum_{i=1}^N(1-c_i)}\prod_{i:c_i>0}F_i(\sqrt{c_i}\widetilde{C}_iw)^{c_i}
$$If we denote by $g_i(z):=F_i(\sqrt{c_i}z)$, then 
$$
\int_\mathbb{C}g_i(z)dz=c_i^{-1}.
$$Then we apply geometric Brascamp-Lieb inequality (see \cite{barthe2008gaussian},\cite{bennett2008brascamp}) to get 
$$
p_U(y)\leq K^{N-\sum_ic_i}\prod_{i:c_i>0}
c_i^{-c_i}.$$
Finally, use $a\log\frac{1}{a}\leq 1-a\quad\forall 0<a<1$ and $\sum_ic_i=N-r$ to conclude.\end{proof}

 \begin{proof}[\proofname\ of Corollary \ref{corollary6.8}] Denote by $\Gamma=AA^*$ which is positive definite. Take $\widehat{U}=\Gamma^{-1/2}A$. Then $\widehat{U}\widehat{U}^*=I_{\mathbb{C}^d}$. By Lemma \ref{lemma1519}, $\widehat{U}X$ has a density $P_{\widehat{U}}$ on $\mathbb{C}^d$ with density bounded by $(eK)^d$. 

 Now $AX=\Gamma^{1/2}\widehat{U}X$, and the complex-linear map $\Gamma^{1/2}:\mathbb{C}^d\to\mathbb{C}^d$ has real Jacobian 
 $$
|\det_\mathbb{C}(\Gamma^{1/2})|^2=\det_\mathbb{C}\Gamma=\det_\mathbb{C}(AA^*),
 $$so by change of variables $AX$ has a density $p_A$ satisfying 
 $$
\|P_A\|_{L^\infty(\mathbb{C}^d)}\leq \frac{(eK)^d}{\det_\mathbb{C}(AA^*)}.
 $$Our assumption on $A$ implies $\det_\mathbb{C}(AA^*)\geq\sigma^d$, so $AX$ has density bounded by $(\frac{eK}{\sigma})^d$. Finally we use
 $$
\mathbb{P}\{\|AX-y\|_2\leq r\}\leq \|P_A\|_\infty\operatorname{Vol}_{\mathbb{C}^d}(B(0,r)),
 $$
 and $\operatorname{Vol}_{\mathbb{C}^d}(B(0,r))\leq (Cr^2)^d.$
     
 \end{proof}

 \begin{proof}[\proofname\ of Fact \ref{facts1518}] For the second moment, 
 $$
\mathbb{E}\|PX\|_2^2=\mathbb{E}\langle PX,X\rangle=\operatorname{tr}(P\mathbb{E}XX^*)=k.
 $$For higher moments, we write in coordinates as 
 $$
PX=\sum_{i=1}^N\xi_i Pe_i,
 $$which is a sum of independent mean-zero random variables. By Rosenthal \cite{rosenthal1970subspaces}, we have 
 $$
\mathbb{E}\left\|\sum_i\xi_i Pe_i\right\|_2^Q\leq C_Q\left[\left(\sum_i\mathbb{E}|\xi_i|^2\|Pe_i\|_2^2\right)^{Q/2}+\sum_i\mathbb{E}|\xi_i|^Q\|Pe_i\|_2^Q
\right].
 $$Then we use
 $
\sum_i\|Pe_i\|_2^2=k,
 $ and since $P$ is orthogonal projection, $\|Pe_i\|_2\leq 1$, so 
 $$\sum_i\|Pe_i\|_2^Q\leq\sum_i\|Pe_i\|_2^2=k,$$ and thus 
 $$
\mathbb{E}\|PX\|_2^Q\leq C_Q(k^{Q/2}+M_Qk).
 $$
 \end{proof}

\begin{proof}[\proofname\ of Lemma \ref{lemmaline1725}]
We only give a sketch. Write the centered random part for the Hermitized linearization via 
$$
Y_{\rho,N}=\frac{1}{\sqrt{N}}\sum_{\ell=1}^n\sum_{a,b=1}^N(g_{\ell ab}^{(\rho)}B_{\ell ab}+\overline{g_{\ell ab}^{(\rho)}}B_{\ell ab}^*),
$$where we denote by $B_{\ell ab}=\begin{bmatrix}
    0&H_\ell\otimes E_{ab}\\0&0
\end{bmatrix}=\mathcal{B}_\ell\otimes E_{ab}$. Here $\mathcal{B}_\ell\in M_{2m}(\mathbb{C})$ and $a,b\in[N]$.
Consider the covariance maps 
$$
C_\rho(M)=\mathbb{E}Y_{\rho,N}MY_{\rho, N},\quad C_0(M)=\mathbb{E}Y_{0,N}MY_{0,N},
$$then these two maps only differ through the coupling $\mathbb{E}(g_{\ell ab}^{(\rho)})^2=\rho,\quad \mathbb{E}(\overline{g_{\ell ab}^{(\rho)}})^2=\bar{\rho}$. Then
$$
C_\rho(M)-C_0(M)=\frac{1}{N}\sum_{\ell,a,b}[\rho B_{\ell ab}M B_{\ell ab}+\bar{\rho}B_{\ell ab}^* MB_{\ell ab}^*].
$$
Then we use the identity $\sum_{a,b=1}^N E_{ab}AE_{ab}=A^T$ for any $A\in M_N(\mathbb{C})$ and write
$$\sum_{a,b}B_{\ell ab}M B_{\ell ab}=\sum_{a,b}(\mathcal{B}_\ell\otimes E_{ab})M(\mathcal{B}_\ell\otimes E_{ab})=(\mathcal{B}_\ell\otimes I_N)M^\Gamma(\mathcal{B}_\ell\otimes I_N).
$$Here $M^\Gamma$ is the partial transpose of $M$, which means we only take the transpose of each block of size $N$ but do not change the location of the $m^2$ blocks.
Since the dimension $m$ is bounded, partial transpose has bounded operator norm, so we can find $C_\mathfrak{p}>0$ such that 
$$
\|C_\rho-C_0\|_{\infty\to\infty}\leq \frac{C_\mathfrak{p}|\rho|}
{N}.$$
The mean of $\mathcal{H}_{\rho,N}^z$
and $\mathcal{H}_{0,N}^z$ are equal. Applying the Gaussian resolvent comparison in \cite{brailovskaya2024universality}, Lemma 8.3, we get 
$$\|\mathbb{E}(wI-\mathcal{H}_{\rho,N}^z)^{-1}-\mathbb{E}(wI-\mathcal{H}_{0,N}^z)^{-1}\|\leq \frac{C_\mathfrak{p}|\rho|}{N(\Im w)^3}
.$$
This completes the proof.
\end{proof}

\section*{Acknowledgment}
The author thanks Charles Bordenave and Ping Zhong for introducing the author to this very interesting problem.

\section*{Funding}
The author is supported by a fellowship from IAS provided by the S.S. Chern Foundation for Mathematical Research Fund and the Fund for Mathematics.

\printbibliography

\end{document}